\documentclass[11pt,a4paper]{article}
\usepackage{amsmath,amsthm,amssymb}

\renewcommand{\phi}{\varphi}
\renewcommand{\emptyset}{\varnothing}
\newcommand{\Zd}{\mathbb{Z}^d}
\renewcommand{\Pr}{\mathbf{P}}

\theoremstyle{plain}
\newtheorem{theorem}{Theorem}
\newtheorem{lemma}{Lemma}
\newtheorem{proposition}{Proposition}
\newtheorem{corollary}{Corollary}

\theoremstyle{definition}

\theoremstyle{remark}
\newtheorem{remark}{Remark}

\newcommand{\eqn}[2]{\begin{equation}\label{#1}#2\end{equation}}
\newcommand{\eqnst}[1]{\begin{equation*}#1\end{equation*}}
\newcommand{\eqnspl}[2]{\begin{equation}\begin{split}\label{#1}%
    #2\end{split}\end{equation}}
\newcommand{\eqnsplst}[1]{\begin{equation*}\begin{split}%
    #1\end{split}\end{equation*}}

\newcommand{\Z}{\mathbb Z}

\newcommand{\E}{\mathbb E}

\def\caA{\mathcal{A}}
\def\caR{\mathcal{R}}
\def\caS{\mathcal{S}}

\def\caT{\mathcal{T}}
\def\caP{\mathcal{P}}
\def\caF{\mathcal{F}}
\def\caG{\mathcal{G}}
\def\caX{\mathcal{X}}
\def\La{\Lambda}
\def\Om{\Omega}

\newcommand{\tOm}{\widetilde{\Om}}

\def\discr{{\mathrm{discr}}}
\def\es{\emptyset}

\def\eps{\varepsilon}
\def\LE{\mathrm{LE}}

\def\bB{\mathrm{\bf{B}}}

\def\Unif{{\mathrm{Unif}}}
\def\ext{{\mathrm{ext}}}
\def\ord{\mathrm{ord}}
\def\diss{\mathrm{diss}}

\def\CMPsh{Commun.~Math.~Phys.}

\def\AOPsh{Ann.~Probab.}

\begin{document}

\title{{\bf Rate of convergence estimates for the zero 
dissipation limit in Abelian sandpiles}}

\author{
Antal A.~J\'arai
\footnote{Department of Mathematical Sciences,
University of Bath,
Claverton Down, Bath, BA2 7AY, United Kingdom, 
Email: {\tt A.Jarai@bath.ac.uk}}%
}

\maketitle

\footnotesize
\begin{quote}
{\bf Abstract}: We consider a continuous height version
of the Abelian sandpile model with small amount of 
bulk dissipation $\gamma > 0$ on each toppling,
in dimensions $d = 2, 3$. 
In the limit $\gamma \downarrow 0$, we give 
a power law upper bound, based on coupling, 
on the rate at which the 
stationary measure converges to the discrete
critical sandpile measure. The proofs are based
on a coding of the stationary measure by 
weighted spanning trees, and an analysis
of the latter via Wilson's algorithm.
In the course of the proof, we prove an estimate
on coupling a geometrically killed 
loop-erased random walk to an unkilled
loop-erased random walk. 
\end{quote}
\normalsize

\vspace{12pt}

\section{Introduction}
\label{sec:intro}

In the paper \cite{JRS10} a continuous height version of the
Abelian sandpile model was studied. The reason for interest
in that model is that it allows an arbitrarily small amount
of dissipation on every toppling, and hence yields a natural
family of subcritical models approximating the discrete
(critical) sandpile. In the present paper we study the
dissipative models on $\Z^d$, $d = 2,3$, and give a power law
upper bound on the rate at which the stationary measure of the
dissipative model converges to the critical sandpile measure.
Our proof also applies in $d=4$, but would only yield a
logarithmic bound. Our methods break down for $d \ge 5$.
Hence it remains an open problem to
give a power law bound in dimensions $d \ge 4$.

Let us first recall the definition of the discrete Abelian 
sandpile model. See the survey by Redig \cite{Redig} 
for general background, and the paper \cite{HLMPPW}
for a nice introduction to the basic facts. 
Let $\Lambda \subset \Z^d$ be finite. We define the
set of stable configurations in $\La$ as 
\eqnst
{ \Om^{\discr}_\La 
  = \{ 0, 1, \dots, 2d-1 \}^\La. }
We also consider the set of all non-negative configurations
$\caX^\discr_\La = \{ 0, 1, \dots \}^\La$. Consider the 
\emph{toppling matrix}
\eqnst
{ \Delta_{xy}
  = \begin{cases}
    2d & \text{if $x = y$;}\\
    -1 & \text{if $x \sim y$;}\\
    0  & \text{otherwise.}
    \end{cases}
    \qquad x,y \in \La, }
where $x \sim y$ denotes that $x$ and $y$ are adjacent 
on the $\Z^d$ lattice. Let $\eta \in \caX^\discr_\La$.
If $\eta_x \ge 2d$, we say that $x$ can be \emph{legally toppled},
and the result of toppling $x$ is the new configuration
$T^{(x)} \eta$ defined by 
\eqnst
{ (T^{(x)} \eta)_y
  := \eta_y - \Delta_{xy}, \qquad y \in \La. }
Note that if the toppling was legal, $T^{(x)} \eta \in \caX^\discr_\La$.
If a finite sequence of legal topplings has a stable result,
it is called the \emph{stabilization of $\eta$}, and is denoted
$\caS^\discr_\La(\eta)$. It is well-known \cite{Dhar,Redig,HLMPPW}
that stablization is well-defined as a map 
$\caS^\discr_\La : \caX^\discr_\La \to \Om^{\discr}_\La$. 

The dynamics of the model is defined as a Markov chain
$(\eta(n))_{n \ge 0}$ with state space $\Om^{\discr}_\La$. 
Let $X_1, X_2, \dots$ be i.i.d.~with a fixed distribution
$\{ p(x) \}_{x \in \La}$ on $\La$, where $p(x) > 0$, $x \in \La$. 
At time $n \ge 1$ a new particle is added at $X_n$, 
and the configuration is stabilized. That is,
$\eta(n) := \caS^\discr_\La(\eta(n-1) + \delta_{X_n, \cdot})$,
where $\delta_{xy}$ is Kronecker's delta.
It is well-known \cite{Dhar,Redig} (see also 
\cite[Corollary 2.16]{HLMPPW}) that the Markov chain has
a single recurrent class $\caR^\discr_\La$, and the
stationary distribution, denoted $\nu_\La$, 
is uniform on $\caR^\discr_\La$ (irrespective of the choice of $p$). 

We now give the definition of the continuous height dissipative 
model in finite volume $\La \subset \Z^d$. Let $\gamma \ge 0$ 
be a real parameter, and define the set of stable configurations 
$\Om^{(\gamma)}_\La := [0, 2d + \gamma)^\La$. Let us also define the 
set of all non-negative configurations 
$\caX_\La := [0, \infty)^\La$. Consider the toppling matrix 
given by:
\eqnst
{ \Delta^{(\gamma)}_{xy}
  = \begin{cases}
    2d + \gamma & \text{if $x = y$;} \\
    -1          & \text{if $x \sim y$;} \\
    0           & \text{otherwise.} 
    \end{cases} 
    \qquad x,y \in \La. }
Let $\eta \in \caX_\La$. If $\eta_x \ge 2d + \gamma$ for some 
$x \in \La$, then we say that 
\emph{$x$ can be $\gamma$-legally toppled}, and
the result of toppling $x$ is the new configuration
$T^{(\gamma)}_x \eta$ defined by 
$(T^{(\gamma)}_x \eta)_y := \eta_y - \Delta^{(\gamma)}_{xy}$, $y \in \La$.
If a finite sequence of $\gamma$-legal topplings has a 
stable result, it is called the \emph{$\gamma$-stabilization} 
of $\eta$, and is denoted $\caS^{(\gamma)}_\La(\eta)$. 
By arguments similar to the discrete case, it is not difficult 
to show that $\caS^{(\gamma)}_\La$ is well-defined as a 
map $\caS^{(\gamma)}_\La : \caX_\La \to \Om^{(\gamma)}_\La$. 
see for example \cite{Dhar} or \cite[Appendix B]{G93}.

The dynamics of the dissipative model is defined as a Markov chain
$(\eta(n))_{n \ge 0}$ with state space $\Om^{(\gamma)}_\La$.
Let $X_1, X_2, \dots$ be i.i.d.~with distribution $p$. 
At time $n \ge 1$, unit height is added at $X_n$
and the configuration is stabilized according to the toppling matrix 
$\Delta^{(\gamma)}$, that is, 
$\eta(n) = \caS^{(\gamma)}_\La(\eta(n-1) + \delta_{X_n,\cdot})$.
Similarly to the discrete model, a set of recurrent configurations 
$\caR^{(\gamma)}_\La$ can be defined, and Lebesgue measure
on $\caR^{(\gamma)}_\La$ is invariant for the 
dynamics \cite[Section 2.2]{JRS10}.
We denote the invariant probability measure by $m^{(\gamma)}_\La$.

It was shown in \cite[Theorem 1]{AJ04} and \cite[Appendix]{JR08}
that for any $d \ge 2$, the measures $\nu_\La$ weakly converge 
to a limit $\nu$ on the space 
$\Om^{\discr} := \{ 0, 1, \dots, 2d-1 \}^{\Z^d}$, as 
$\La \uparrow \Z^d$ (the limit can be taken along any 
sequence of $\La$'s that exhaust $\Z^d$).
It is quite clear that the measure $m^{(0)}_\La$ should be 
closely related to $\nu_\La$. Indeed, based on results
further explained in Section \ref{sec:spanning}, one also
easily gets that $m^{(0)}_\La$ has a unique weak limit 
$m^{(0)}$.

It was further shown in \cite[Lemma 4]{JRS10}, that for
all $d \ge 2$ and any $\gamma > 0$,
the weak limit $m^{(\gamma)}_\La \Rightarrow m^{(\gamma)}$ exists
for all $d \ge 2$.
It was also shown in \cite[Proposition 4]{JRS10} 
that as $\gamma \downarrow 0$, $m^{(\gamma)} \Rightarrow m^{(0)}$. 
The goal of the present paper is to prove the following theorem.
Let $\Om^{(\gamma)} := [0,2d+\gamma)^{\Z^d}$ and 
let $\Om := [0, 2d)^{\Z^d}$.

\begin{theorem}
\label{thm:main}
Suppose that the cylinder event $E \subset \Omega$ 
depends only on the heights in $B(k) = [-k,k]^d \cap \Z^d$. 

(1) If $d = 3$, there exist constants 
$C < \infty$, $\eta > 0$ such that for all $\gamma < 1$
\eqnst
{ \left| m^{(\gamma)}(E) - m^{(0)}(E) \right|
  \le C k^2 \gamma^\eta + C k^5 (\log k) \gamma. }

(2) If $d = 2$, there exist constants $c_0 > 0$ and 
$C, C_0 < \infty$ such that for all $\gamma \le c_0 k^{-C_0}$
we have
\eqn{e:mainthm}
{ \left| m^{(\gamma)}(E) - m^{(0)}(E) \right|
  \le C k^{21/23} \gamma^{1/46-o(1)}, }
with $o(1)$ denoting a positive quantity that approaches 
$0$ as $\gamma \to 0$.
\end{theorem}

\begin{remark}
For certain special events, a better exponent of $\gamma$ 
was obtained in \cite[Proposition 5]{JRS10}. There the estimate
for $d \ge 3$ is of the form $C(k) \gamma$ and 
for $d = 2$ of the form $C(k) \gamma \log(1/\gamma)$.
A natural question is whether these represent the
precise rate of convergence for all events.
\end{remark}

\begin{remark}
A suitable value of $C_0$ is determined from the proof.
For $\gamma \le c (2k)^{-23}$ we have the somewhat worse 
bound $C k^{3/2} \gamma^{1/46-o(1)}$.
\end{remark}

\section{Discretized heights and spanning trees}
\label{sec:spanning}

Our proof of Theorem \ref{thm:main} is based on a certain 
``descretization'' of the measure $m^{(\gamma)}$, and a coding
of the discretized measure by weighted spanning trees.
These tools were already introduced in \cite{JRS10}, however, 
here we will need somewhat more detailed properties of the 
coding than in \cite{JRS10}. The form of the coding follows 
from the proof of \cite[Theorem 1]{AJ04}. The coding is the main link 
between Theorem \ref{thm:main} and the loop-erased random
walk estimates that make up most of the proof. We explain 
the coding in detail in this Section.

\subsection{Allowed configurations}
\label{ssec:allowed}

Let $\La \subset \Z^d$ be finite. The sets $\caR^\discr_\La$
and $\caR^{(\gamma)}_\La$ admit the following description. 
Let $\eta \in \Omega^{\discr}_\La$ or $\eta \in \Omega^{(\gamma)}_\La$
with some $\gamma \ge 0$. For $\es \not= W \subset \La$
write $\eta_W$ for the restriction of the configuration
$\eta$ to $W$. We say that $\eta_W$ is 
a \emph{forbidden subconfiguration (FSC)}, if 
\eqnst
{ \eta_y 
  < \left| \{ z \in W : z \sim y \} \right|, \qquad y \in W, }
where $|A|$ denotes the cardinality of the set $A$.
A configuration $\eta \in \Om^{\discr}_\La$ or 
$\eta \in \Om^{(\gamma)}_\La$ is called \emph{allowed}, 
if there is no
$\es \not= W \subset \La$ such that $\eta_W$ is a FSC.
It was shown in \cite{Dhar} that all elements of 
$\caR^{\discr}_\La$ are allowed, and it was shown in 
\cite{MD92}, using the ``Burning Test'' explained below, 
that all allowed configurations 
in $\Om^{\discr}_\La$ are in $\caR^{\discr}_\La$. 
The analogous statement is also true in the
continuous case: it follows from \cite[Appendix E]{G93}
that $\caR^{(\gamma)}_\La$ consists precisely of
the allowed configurations of $\Om^{(\gamma)}_\La$.

\subsection{Discretization and the uniformity property}
\label{ssec:discretize}

The characterization in terms of allowed configurations 
shows that $m^{(\gamma)}_\La$ can be completely understood
in terms of a discrete measure. Let
\eqnst
{ \Om^{\discr,*}_\La 
  := \{ 0, 1, \dots, 2d-1, 2d \}^\La, }
and define the map $\psi_\La : \caX_\La \to \Om^{\discr,*}_\La$
via 
\eqn{e:discretize}
{ \xi_x 
  := (\psi_\La (\eta))_x 
  = \begin{cases}
    h  & \text{if $h \le \eta_x < h+1$ for some 
       $h \in \{0, 1, 2, \dots, 2d-1 \}$;} \\
    2d & \text{if $\eta_x \ge 2d$. }
    \end{cases} }
We can also view $\psi_\La$ as a map from $\Om^{(\gamma)}_\La$
in a natural way, for any $\gamma \ge 0$.
It follows directly from the definition of FSCs 
that for any $\gamma \ge 0$ and any $\eta \in \Om^{(\gamma)}_\La$, 
we have the equivalence:
\eqnst
{ \text{$\eta$ is allowed} \qquad \Longleftrightarrow \qquad
  \text{$\xi = \psi_\La(\eta)$ is allowed.} }
This together with the fact that $m^{(\gamma)}_\La$ is 
normalized Lebesgues measure, implies the following statement.
\eqn{e:uniformity}
{ \parbox{12cm}{Under the measure $m^{(\gamma)}_\La$,
  and given the value of $\xi = \psi_\La(\eta)$,
  the variables $(\eta_x)_{x \in \La}$ are conditionally independent
  with conditional distributions:
  $\eta_x \sim \Unif(h,h+1)$ when $\xi_x = h$, $h = 0, 1, \dots, 2d-1$,
  and $\eta_x \sim \Unif(2d,2d+\gamma)$ when $\xi_x = 2d$.} }
We will denote by $\nu^{(\gamma)}_\La$ the image of
$m^{(\gamma)}_\La$ under the map $\psi_\La$. 
By the definition of $\psi_\La$, the measure $\nu^{(0)}_\La$ 
concentrates on $\Om^\discr_\La$, and in fact coincides with 
$\nu_\La$, by the characterization in terms of allowed
configurations. 

We now proceed to describe $\nu_\La$ and $\nu^{(\gamma)}_\La$
when $\gamma > 0$. Let us write 
\eqnst
{ \caA_\La
  = \{ \eta \in \Om^{\discr,*}_\La : \text{$\eta$ is allowed} \}. } 

We first observe that $\nu^{(\gamma)}_\La$ obeys a certain
weighting depending on $\gamma$. 
For $\xi \in \caA_\La$, let 
$H(\xi) = \left| \{ x \in \La : \xi_x = 2d \} \right|$.
Since $m^{(\gamma)}_\La$ is normalized Lebesgue measure, 
\eqref{e:uniformity} implies that
\eqn{e:weights}
{ \nu^{(\gamma)}_\La(\xi)
  = c \gamma^{H(\xi)}, \quad \xi \in \caA_\La,  }
for some constant $c = c(\gamma,\La)$. On the other hand,
as stated before, $\nu_\La$ is uniform on 
$\caA_\La \cap \Om^\discr_\La$. 

\subsection{The Burning Test}
\label{ssec:burning}

There is a simple algorithm, the Buring Test \cite{Dhar},
that checks if a configuration is allowed or not.
There is some flexibility in setting up the algorithm.
We choose one here that will work for both 
$\xi \in \Om^{\discr}_\La$ and $\xi \in \Om^{\discr,*}$. 
Set $U_0 := \La$. We call $U_0$ the set of 
vertices \emph{unburnt at time $0$}. We define
\eqnsplst
{ B_1 
  &:= \left\{ x \in U_0 : \xi_x = 2d \right\} \\
  U_1 
  &:= \La \setminus B_1. }
We inductively define for $i \ge 2$:
\eqnsplst
{ B_i 
  &:= \left\{ x \in U_{i-1} : 
     \eta_x \ge \left| \{ y \in U_{i-1} : y \sim x \} \right| \right\}; \\
  U_i
  &:= U_{i-1} \setminus B_i. }
We call $B_i$ and $U_i$ the set of vertices 
\emph{burning at time $i$} and \emph{unburnt at time $i$}, 
respectively.
We say that the algorithm \emph{terminates}, if 
$U_i = \es$ for some $i$. Using the definition of FSCs,
it can be shown by induction on $i$ that there cannot be 
any FSC containing a vertex in $B_i$, $i \ge 1$, 
and therefore, if the algorithm 
terminates, then $\eta$ is allowed. On the other hand,
if the algorithm does not terminate, and 
$\es \not= U_i = U_{i+1} = \dots =: U$, then $\eta_U$
is an FSC. Hence the algorithm terminates if and only
if $\eta$ was allowed. 

\subsection{The Majumdar-Dhar bijection with spanning trees}
\label{ssec:spanning}

A useful description of $\nu_\La$ and $\nu^{(\gamma)}_\La$
can be given in terms of weighted spanning trees, 
that we now describe. This bijection was discovered
by Majumdar and Dhar \cite{MD92}, and the extension 
given here to $\nu^{(\gamma)}_\La$ is from \cite{JRS10}.

We now define the multigraph $G_\La = (V(G_\La), E(G_\La))$ 
that will carry the weighted
spanning trees. We first define an infinite graph $G$.
Add a new vertex $\varpi$ to $\Z^d$, so the vertex set
of $G$ is $\Z^d \cup \{ \varpi \}$. The edge set of $G$
consists of: (i) for each $x, y \in \Z^d$ that are adjacent 
in the $\Z^d$ lattice, we place an edge between $x$ and $y$; 
(ii) for each $x \in \Z^d$ we place an edge between $x$ and $\varpi$.
We call the edges of type (i) \emph{ordinary}, and we call 
those of type (ii) \emph{dissipative}.
It will be convenient to refer to the ordinary edge between
$x$ and $y$ as $\ord(x,y) = \ord(y,x)$, and the dissipative
edge between $x$ and $\varpi$ as $\diss(x)$.
For later use, we also define the graph $G^{(0)}$, that is
obtained by removing $\varpi$ and all dissipative edges
from $G$ (in other words, the usual $\Z^d$ lattice). 
The multigraph $G_\La$ is now defined by identifying all vertices
of $G$ in the set $\Z^d \setminus \La$ with $\varpi$, and removing loops.
In $G_\La$ we still call edges dissipative or ordinary,
according to their origin. We also use the notation $\ord(x,y)$
and $\diss(x)$ according to the origin of the edge.
In particular, if $x \in \La$, $y \in \Z^d \setminus \La$
and $x$ and $y$ are adjecent in $\Z^d$, then 
$\ord(x,y)$ denotes an ordinary edge of $G_\La$ between
$x$ and $\varpi$. The usefulness of this is that there 
can be more than one ordinary edges between $x$ and 
$\varpi$, and these are thereby distinguished.
We further define the graph $G_\La^{(0)}$ 
by removing all dissipative edges from $G_\La$ (but
keeping the vertex $\varpi$).

Let $\caT_\La$ denote the set of spanning trees of $G_\La$,
and let $\caT^{(0)}_\La$ denote those spanning trees
that do not contain dissipative edges. The latter is
naturally identified with the set of spanning trees of $G_\La^{(0)}$. 
We now define a map $\sigma_\La : \caA_\La \to \caT_\La$.
Let $\xi \in \caA_\La$. We define the spanning tree
$t = \sigma_\La (\xi)$ in stages. At each stage, we
will connect each vertex in $B_i$ to some vertex 
in $B_{i-1}$, which automatically ensures that there
are no loops. In order to be able to start, 
we define $B_0 = \{ \varpi \}$. Since 
$\cup_{i \ge 0} B_i = \La \cup \{ \varpi \} = V(G_\La)$, the 
construction ensures that we get a spanning tree.

First we connect each $x \in B_1 = H(\xi)$
to $\varpi$, that is we put the edge 
$\diss(x)$ into $t$. Suppose that $i \ge 2$, and 
each vertex in $\cup_{1 \le j < i} B_j$ has been 
connected to an earlier vertex. Let $x \in B_i$, 
and define
\eqnspl{e:data}
{ n_{x,\La} 
  &:= \text{number of ordinary edges between $x$ 
      and $\cup_{0 \le j < i} B_j$}; \\
  P_{x,\La} 
  &:= \{ y \in \Z^d : \text{$|y| = 1$ and $\ord(x,x+y)$ 
      is an edge between $x$ and $B_{i-1}$} \}; \\
  K_{x,\La}
  &:= \{ 2d - n_x, \dots, 2d - n_x + | P_{x,\La} | - 1 \}. }
Suppose that for every set $\es \not= P \subset \{ y \in \Z^d : |y| = 1 \}$
and every set $K \subset \{ 0, 1, \dots, 2d-1 \}$ of the
form $\{ k, k+1, \dots, k + |P| - 1 \}$ an arbitrary 
bijection $\alpha_{P,K} : P \to K$ is fixed. 
The fact that $x$ burns at time $i$ means that we
must have $\xi_x \in K_{x,\La}$: $x$ has $2d - n_x$ 
unburnt neighbours at time $i$, so in order for it to 
burn at time $i$, we must have $\xi_x \ge 2d - n_x$. 
On the other hand, since it did not burn before
time $i$, we must have $\xi_x \le 2d - n_x + |P_{x,\La}| - 1$.
We select the edge $\ord(x,x+\alpha_{P_{x,\La},K_{x,\La}}^{-1}(\xi_x))$
between $x$ and $B_{i-1}$ to be placed in $t$.
This completes the definition of $t = \sigma_\La(\xi)$.

\begin{lemma}[Majumdar-Dhar \cite{MD92}, \cite{JRS10}]\ \\
(i) The map $\sigma_\La$ is a bijection between $\caA_\La$ 
and $\caT_\La$.\\
(ii) The restriction of $\sigma_\La$ to 
$\caA_\La \cap \Om^{\discr}_\La$
is a bijection between this set and $\caT^{(0)}_\La$.
\end{lemma}

\begin{proof}
(i) We show that $\sigma_\La$ is injective. Let 
$\xi^1, \xi^2 \in \caA_\La$, $\xi^1 \not= \xi^2$, and
let $t^1 := \sigma_\La(\xi^1)$, $t^2 := \sigma_\La(\xi^2)$.
If $H(\xi^1) = B_1(\xi^1) \not= B_1(\xi^2) = H(\xi^2)$,
then $t^1$ and $t^2$ differ in at least 
one dissipative edge. Hence we may assume that 
$B_1(\xi^1) = B_1(\xi^2)$, and this implies that 
$\xi^1 = \xi^2$ on $B_1(\xi^1) = B_1(\xi^2)$.
Let $i \ge 2$ be the smallest index such that either 
$B_i(\xi^1) \not= B_i(\xi^2)$ or there exists 
$x \in B_i(\xi^1) = B_i(\xi^2)$ with $\xi^1_x \not= \xi^2_x$.
If such index did not exist, we would get by induction
on $i$ that $\xi^1 = \xi^2$ on 
$\cup_{i \ge 1} B_i(\xi^1) = \cup_{i \ge 1} B_i(\xi^2) = \La$,
a contradiction. By the choice of $i$, we have
\eqn{e:equality}
{ \text{$B_j(\xi^1) = B_j(\xi^2)$ for $1 \le j \le i-1$.} }
If $B_i(\xi^1) \not= B_i(\xi^2)$, then pick a vertex
$x$ in the symmetric difference. Then by the construction 
of $\sigma_\La$, in one of $t^1$ and $t^2$ there is an edge
from $x$ to $B_{i-1}(\xi^1) = B_{i-1}(\xi^2)$ 
and there is no such edge in the other, so 
$t^1 \not= t^2$. Suppose therefore
that $B_i(\xi^1) = B_i(\xi^2)$, but there exists
$x \in B_i(\xi^1) = B_i(\xi^2)$ such that $\xi^1_x \not= \xi^2_x$.
By the equality \eqref{e:equality}, we have
$n_{x,\La}(\xi^1) = n_{x,\La}(\xi^2)$, 
$P_{x,\La}(\xi^1) = P_{x,\La}(\xi^2)$, and hence also
$K_{x,\La}(\xi^1) = K_{x,\La}(\xi^2)$. 
However, since $\xi^1_x \not= \xi^2_x$ we have
$\alpha_{P_{x,\La},K_{x,\La}}^{-1}(\xi^1_x) \not= 
\alpha_{P_{x,\La},K_{x,\La}}^{-1}(\xi^2_x)$, and
therefore the edge between $x$ and $B_{i-1}$ 
is different in $t^1$ and $t^2$.
This completes the proof of injectivity.

We now show that $\sigma_\La$ is surjective. In the course
of doing so, we find the inverse map 
$\sigma_\La^{-1} =: \phi_\La : \caT_\La \to \caA_\La$.
First we note that for any $\xi \in \caA_\La$, 
the sets $B_0, B_1, \dots$ and the data in \eqref{e:data} 
can be easily expressed in terms of $t = \sigma_\La(\xi)$ 
as well. Namely, let 
$d_t(\cdot,\cdot)$ denote graph distance in the 
tree $t$. Then due to the construction of $t$, we have
\eqnspl{e:B's}
{ B_0 
  &= \{ \varpi \}; \\
  B_1
  &= \{ x \in \La : \diss(x) \in t \}; \\
  B_i 
  &= \{ x \in \La : d_t(B_0 \cup B_1, x) = i-1 \}, 
    \qquad i \ge 2. }
Since this expresses $B_0, B_1, \dots$ in terms 
of $t$, the formulas \eqref{e:data} show that 
$n_{x,\La}$, $P_{x,\La}$ and $K_{x,\La}$ are also 
expressed in terms of $t$.
Also, by the definition of $\sigma_\La$, if the unique
edge of $t$ in $P_{x,\La}$ is $\ord(x,x+y)$, then
we have $\xi_x = \alpha_{P_{x,\La}, K_{x,\La}} (y)$.

The above makes it clear what the inverse $\phi_\La = \sigma_\La^{-1}$ 
has to be. Suppose that $t \in \caT_\La$ is given.
We use \eqref{e:B's} to \emph{define} the $B_i$'s and 
for $x \in B_i$, $i \ge 2$, we use \eqref{e:data} 
as the definition of $n_{x,\La}$, $P_{x,\La}$ and $K_{x,\La}$.
For $x \in B_1$, we set $\xi_x = 2d$. 
For $x \in B_i$, $i \ge 2$ let $y_{x,\La} \in \Z^d$ be
such that $\ord(x,x+y_{x,\La})$ is 
the unique edge of $t$ in $P_{x,\La}$, and we set 
$\xi_x = \alpha_{P_{x,\La}, K_{x,\La}} (y_{x,\La})$.
We define $\phi_\La(t) := \xi$. It is clear that 
if $\xi \in \caA_\La$, then $\sigma_\La(\phi_\La(t)) = t$.
What is left to show is that we always have 
$\xi \in \caA_\La$.

We prove that for every $t \in \caT_\La$ we have
$\xi = \phi_\La(t) \in \caA_\La$, 
by applying the Burning Test to $\xi$.
It is immediate that in the first step 
exactly $B_1$ burns as $B_1 = H(\xi)$ by construction.
Suppose now inductively that $i \ge 2$ and we already 
know that at time $1 \le j \le i-1$ exactly 
$B_j$ burns. Let $x \in B_i$. Then due to the 
inductive hypothesis and the definition of
$n_{x,\La}$, $x$ has precisely $2d - n_{x,\La}$ unburnt
neighbours at time $i-1$. Since $\xi_x \in K_{x,\La}$ 
by the definition of $\xi$, we have $\xi_x \ge 2d - n_{x,\La}$
and hence $x$ burns at time $i$. Let now
$x \in B_j$ with $j \ge i+1$. Then by the 
induction hypothesis, $B_{j-1}, B_j, \dots$ are unburnt 
at time $i-1$, and hence the number of unburnt
neighbours of $x$ at time $i-1$ is at least 
$2d - n_{x,\La} + | P_{x,\La} |$. Hence, since
$\xi_x \in K_{x,\La}$, we have 
$\xi_x < 2d - n_x + |P_{x,\La}|$, and therefore
$x$ does not burn at time $i$. This shows that 
at time $i$ precisely the set $B_i$ burns, and
completes the induction. Therefore $\xi$ is allowed,
and we have shown that $\sigma_\La$ is a bijection 
between $\caA_\La$ and $\caT_\La$.

(ii) This second statement of the Lemma is now
clear from the construction.
\end{proof}

For $\gamma \ge 0$, let $\mu^{(\gamma)}_\La$ be the 
probability measure on $\caT_\La$ such that
$\mu^{(\gamma)}(t) = c \gamma^{H(t)}$, where 
$H(t) = \text{number of dissipative edges in $t$}$.
In particular, $\mu_\La = \mu^{(0)}_\La$ is the uniform 
spanning tree measure on $G^{(0)}_\La$.
It is clear from \eqref{e:weights} that for all $\gamma \ge 0$
the measure $\nu^{(\gamma)}_\La$ is precisely the image of 
$\mu^{(\gamma)}_\La$ under $\phi_\La = \sigma_\La^{-1}$.

\medbreak

In Sections \ref{ssec:limits0}, \ref{ssec:limitsgamma} below,
we describe the limits of the measures 
$\mu^{(\gamma)}_\La$, $\nu^{(\gamma)}_\La$, $m^{(\gamma)}_\La$,
$\gamma \ge 0$, as $\La \uparrow \Z^d$. 

\subsection{Infinite volume limits when $\gamma = 0$}
\label{ssec:limits0}

When $\gamma = 0$, it is well known \cite{Pem} 
that the weak limit 
$\mu = \mu^{(0)} = \lim_{\La \uparrow \Z^d} \mu^{(0)}_\La$ 
exists, it is called the Wired Uniform Spanning Forest
measure \cite{BLPS01}. The paper \cite{AJ04} proves that the
weak limit $\nu = \nu^{(0)} = \lim_{\La \uparrow \Z^d} \nu^{(0)}_\La$
also exists, for any $d \ge 2$. In the cases $d = 2, 3$ relevant
for this paper, it follows from the
proof in \cite{AJ04} that the coding in terms
of spanning trees remains true in $\Z^d$, in the
following sense. Recall that the graph $G^{(0)}$
is the $\Z^d$ lattice, and 
let $\tOm^{(0)} := \{0, 1\}^{E(G^{(0)})}$.
Also recall that 
$\Om^\discr = \{0, 1, \dots, 2d-1 \}^{\Z^d}$.
There exists a measurable map 
$\phi^{(0)}: \tOm^{(0)} \to \Om^\discr$ defined
$\mu^{(0)}$-a.e.~such that 
the image of $\mu^{(0)}$ under $\phi$ is $\nu^{(0)}$. 

We now state the form of $\phi^{(0)}$ in the relevant 
cases $d = 2, 3$. This is similar to the
form of $\phi_\La$ in Section \ref{ssec:spanning}. 
First note that $\mu^{(0)}$ concentrates 
on configurations in $\tOm^{(0)}$ that are 
trees with one end \cite[Theorem 4.3]{Pem}.
(We say that a tree has one end, if
any two infinite self-avoiding paths in the tree 
have infinitely many vertices in common.)
Fix $\omega \in \tOm^{(0)}$ that is a tree with
one end, and write $d_\omega(\cdot, \cdot)$ for 
graph distance in the tree $\omega$. For any 
$x \in \Z^d$, let $\pi^{(0)}_x$ denote
the unique self-avoiding path in $\omega$ from 
$x$ to infinity. There exists a unique vertex 
$V_x$ that is on all the paths $\{ \pi^{(0)}_{x+y} \}_{|y| \le 1}$
and is furthest from infinity (along each path).
We define
\eqnspl{e:data3}
{ n^{(0)}_{x} 
  &= \left| \{ y \in \Z^d : |y| = 1,\, 
     d_\omega(x+y,V_x) < d_\omega(x,V_x) \} \right| \\
  P^{(0)}_{x} 
  &:= \left\{ y \in \Z^d : |y| = 1,\, 
     d_\omega(x+y,V_x) = d_\omega(x,V_x)-1 \right\} \\
  K^{(0)}_{x}
  &:= \{ 2d - n^{(0)}_x, \dots, 2d - n^{(0)}_x + | P^{(0)}_{x} | - 1 \} \\
  y^{(0)}_x
  &:= \text{the vertex such that $(x,x+y^{(0)}_x)$ is the 
      first edge of $\pi^{(0)}_x$}. }
Set $\xi_x := \alpha_{P^{(0)}_x,K^{(0)}_x}(y^{(0)}_x)$, $x \in \Z^d$.
It follows from the proof of \cite[Theorem 1]{AJ04}
and the form of $\phi_\La$ in Section \ref{ssec:spanning}
that $\xi =: \phi^{(0)}(\omega)$ is the claimed map. In other
words, $\{ \xi_x \}_{x \in \Z^d}$ has distribution 
$\nu^{(0)}$, if $\omega$ has distribution $\mu^{(0)}$.

Taking into account the uniformity 
property \eqref{e:uniformity}, it is fairly straightforward
to show \cite[Lemma 4]{JRS10} that we also have the weak limit 
$m^{(0)} := \lim_{\La \uparrow \Z^d} m^{(0)}_\La$. Also,
the uniformity property is preserved in this limit.

\subsection{Infinite volume limits when $\gamma > 0$}
\label{ssec:limitsgamma}

We now describe what happens for $\gamma > 0$. Due to 
the monotonicity properties of weighted spanning tree
measures \cite[Sections 4,5]{BLPS01} it follows that 
the weak limit 
$\mu^{(\gamma)} := \lim_{\La \uparrow \Z^d} \mu^{(\gamma)}_\La$
exists for any $\gamma > 0$. 

We will be interested in sampling from $\mu^{(\gamma)}$ via Wilson's 
algorithm \cite{Wilson}. Consider the network random
walk on the graph $G$, where ordinary edges have 
weight $1$ and dissipative edges have weight $\gamma$.
This is the Markov chain with state space 
$\Z^d \cup \{ \varpi \}$ and transition probabilities 
\eqnsplst
{ p_{x,y} 
  &= \frac{1}{2d + \gamma}, 
     \qquad \text{if $x, y \in \Z^d$, $|x-y| = 1$;} \\
  p_{x,\varpi}
  &= \frac{\gamma}{2d + \gamma},
     \qquad \text{if $x \in \Z^d$.} }
The Markov chain is stopped at the first time $\varpi$ is hit.
The analogous network random walk can also be defined on the graph 
$G_\La$, where again, ordinary edges have weight $1$ and
dissipative edges have weight $\gamma$.

If $\rho = [\rho_0, \rho_1, \dots]$ is a finite path in 
$G_\La$ or $G$, the \emph{loop-erasure} $\LE(\rho)$ of $\rho$ 
is defined by chronologically erasing loops from the path 
$\rho$, as they are created.
That is, $\LE(\rho) = [\pi_0, \pi_1, \dots]$, where
$\pi_0 := \rho_0$, and for $i \ge 1$ we inductively define
\eqnsplst
{ s_i 
  &:= \max \{ n \ge 0 : \rho_n = \pi_{i-1} \} \\
  \pi_i
  &:= \rho_{s_i + 1}. }
Note that loop-erasure can also be defined for infinite
paths $\rho$ that visit any vertex only finitely often.
When $(S_n)_{n \ge 0}$ is the network random walk on $G_\La$
or $G$, the loop-erasure is called the
\emph{Loop-Erased Random Walk} (LERW) \cite{Lawler}.

Wilson's algorithm in the case of $G_\La$ can be 
stated as follows. Let $x_1, \dots, x_N$ be an enumeration 
of all vertices in $\La$. Let 
$(S^1_n)_{n \ge 0}, \dots, (S^N_n)_{n \ge 0}$ be independent
network random walks started at $x_1, \dots, x_N$, respectively.
We define a growing sequence of trees 
$\caF_0 \subset \caF_1 \subset \dots$ as follows.
Put $\caF_0 := \varpi$, and for $i \ge 1$ define
inductively
\eqnspl{e:trees}
{ T^i 
  &:= \inf \{ n \ge 0 : S^i(n) \in \caF_{i-1} \} \\
  \caF_i 
  &:= \caF_{i-1} \cup \LE(S^i[0,T^i]). }
Wilson's Theorem \cite{Wilson} implies that $\caF_N$ 
has distribution $\mu^{(\gamma)}_\La$, irrespective
of the chosen enumeration of the vertices.
The algorithm also applies to $G^{(0)}_\La$ (this 
can be obtained by setting $\gamma = 0$, so that 
dissipative egdes are never traversed), and
it produces a sample from $\mu^{(0)}_\La$.

Wilson's algorithm in the case of $G$ is similar. 
We start with an enumeration $x_1, x_2, \dots$ of all
vertices of $\Z^d$, and define the growing sequence
of trees $\caF_0 \subset \caF_1 \subset \dots$ as in
\eqref{e:trees}. Put $\caF = \cup_{i=0}^\infty \caF_i$.
By an argument similar to \cite[Theorem 5.1]{BLPS01}
it follows that $\caF$ is distributed as $\mu^{(\gamma)}$.
Note that due to transience, Wilson's algorithm
also makes sense on $G^{(0)}$ when $d = 3$, and again
by \cite[Theorem 5.1]{BLPS01} it produces a sample
from $\mu^{(0)}$.

We now describe the relationship between $\mu^{(\gamma)}$ 
and $\lim_{\La \uparrow \Z^d} \nu^{(\gamma)}_\La$. 
Let $\tOm := \{ 0, 1 \}^{E(G)}$. 
It follows from Wilson's algorithm that when 
$\gamma > 0$, for $\mu^{(\gamma)}$-a.e.~configuration,
for every $x \in \Z^d$ there is a finite path from 
$x$ to $\varpi$.
Fix an $\omega \in \tOm$ with this property, and let
$\pi_x$ denote the unique self-avoiding 
path in $\omega$ from $x$ to $\varpi$. 
We define
\eqnspl{e:data4}
{ n_{x} 
  &= \left| \{ y \in \Z^d : |y| = 1,\, 
     d_\omega(x+y,\varpi) < d_\omega(x,\varpi) \} \right| \\
  P_{x} 
  &:= \left\{ y \in \Z^d : |y| = 1,\, 
     d_\omega(x+y,\varpi) = d_\omega(x,\varpi)-1 \right\} \\
  K_{x}
  &:= \{ 2d - n_x, \dots, 2d - n_x + | P_{x} | - 1 \} \\
  y_x
  &:= \text{the vector such that $(x,x+y_x)$ is the 
      first edge of $\pi_x$}. }
Set $\xi_x := \alpha_{P_x,K_x}(y_x)$, $x \in \Z^d$, and
define the map 
$\phi : \tOm \to \Om^{\discr,*} = \{0, 1, \dots, 2d-1, 2d \}^{\Z^d}$
by $\xi =: \phi(\omega)$. 
The following lemma is a more explicit version
of \cite[Lemma 3]{JRS10}.

\begin{lemma}
\label{lem:convgamma}\ \\
(i) The weak limit $\nu^{(\gamma)} := \lim_{\La \uparrow \Z^d} \nu^{(\gamma)}_\La$
exists for all $\gamma > 0$.\\
(ii) The image of $\mu^{(\gamma)}$ under $\phi$ is $\nu^{(\gamma)}$
for all $\gamma > 0$.
\end{lemma}

\begin{proof}
(i) Let $x \in \Z^d$, and let $x_1, \dots, x_{2d+1}$ 
be an enumeration of $\{ x + y \in \Z^d : |y| \le 1 \}$.
Using Wilson's algorithm, we get that for any $x \in \Z^d$,
the joint law of the paths $\{ \pi_{x+y,\La} \}_{|y| \le 1}$
under $\mu^{(\gamma)}_\La$ converges to the joint law of the paths 
$\{ \pi_{x+y} \}_{|y| \le 1}$ under $\mu^{(\gamma)}$, 
as $\La \uparrow \Z^d$. 
It is clear from the definitions of the maps $\phi_\La$ 
and $\phi$ that $\xi_{x,\La}$ and $\xi_x$ only depend on these
paths. Therefore, it follows that the law of
$\xi_{x,\La}$ under $\nu^{(\gamma)}_\La$ converges to 
the law of $\xi_x$ under $\mu^{(\gamma)} \circ \phi^{-1}$.

By essentially the same argument, we obtain that for any 
finite $B \subset \Z^d$, the joint law of $(\xi_{x,\La})_{x \in B}$
under $\nu^{(\gamma)}_\La$ converges to the joint law
of $(\xi_x)_{x \in B}$ under $\mu^{(\gamma)} \circ \phi^{-1}$.
For this we only need to observe that the paths 
$\{ \pi_{x+y,\La} \}_{x \in B,\, |y| \le 1}$ determine 
$(\xi_{x,\La})_{x \in B}$. This proves the weak convergence
statement.

(ii) This follows directly from the form of
$\phi_\La$ in Section \ref{ssec:spanning} and the
convergence $\pi_{x,\La} \Rightarrow \pi_{x}$.
\end{proof}

Lemma \ref{lem:convgamma} and the uniformity 
property \eqref{e:uniformity} implies the convergence of the
measures $m^{(\gamma)}_\La$. Analogously to the finite
$\La$ case define $\caX := \{ 0, 1, \dots \}^{\Z^d}$, and
let $\psi : \caX \to \Om^{\discr,*} = \{0, 1, \dots, 2d-1, 2d \}^{\Z^d}$ 
be defined by the same formulae as in \eqref{e:discretize}. 

\begin{corollary}
\label{cor:mlimit}\cite[Lemma 4]{JRS10} \\
(i) For any $\gamma \ge 0$ the weak limit 
$m^{(\gamma)} := \lim_{\La \uparrow \Z^d} m^{(\gamma)}_\La$
exists and the image of $m^{(\gamma)}$ under $\psi$
is $\nu^{(\gamma)}$.\\
(ii) The measure $m^{(\gamma)}$ still satisfies the
uniformity property \eqref{e:uniformity}
\end{corollary}

Finally, it was shown in \cite[Proposition 4]{JRS10}
that as $\gamma \downarrow 0$, $\nu^{(\gamma)} \Rightarrow \nu^{(0)}$
and correspondingly $m^{(\gamma)} \Rightarrow m^{(0)}$.
The rest of the paper will give a quantitative 
version of this statement, in proving Theorem \ref{thm:main}.

\section{Strategy of the proof}
\label{sec:strategy}

Due to the uniformity property of Corollary \ref{cor:mlimit}(ii),
it is sufficient to prove Theorem \ref{thm:main}
for the measures $\nu^{(\gamma)}$ and $\nu^{(0)}$, in place
of $m^{(\gamma)}$ and $m^{(0)}$.

Let $E \subset \Om^\discr$ be a cylinder event depending
on heights in $B(k)$, and let 
$w_1, \dots, w_K$ be a list of all the vertices in 
$B(k) \cup \partial_\ext B(k)$, where 
$\partial_\ext B(k) = \{ y \in B(k)^c : 
\text{$\exists z \in B(k)$ $y \sim z$} \}$. We know that 
$\mu^{(0)}$-a.s.~there exists $V \in \Z^d$ common to all paths
$\pi^{(0)}_{w_i}$, $1 \le i \le K$. We select $V$ to be the 
earliest such vertex with respect to graph distance from $B(k)$. 

The key property of the correspondence $\phi^{(0)}$ given 
in Section \ref{ssec:limits0} is that the height configuration 
$( \xi_x )_{x \in B(k)} = ( \phi^{(0)}(\omega)_x )_{x \in B(k)}$
only depends on the portion of the paths $\pi^{(0)}_{w_i}$
up to the vertex $V$.

When $\gamma$ is small, with high $\mu^{(\gamma)}$-probability 
there will be a $V^{(\gamma)} \in \Z^d$ such that all the paths
$\pi^{(\gamma)}_{w_i}$ meet at $V^{(\gamma)}$ before they reach $\varpi$.
Our goal is to couple $\mu^{(\gamma)}$ and $\mu^{(0)}$ in such 
a way that with high probability $V = V^{(\gamma)}$ and 
$\pi^{(\gamma)}_{w_i} = \pi^{(0)}_{w_i}$ up to the vertex $V^{(\gamma)} = V$.

Consider first $d = 3$. A natural coupling is given by 
applying Wilson's algorithm with the same 
random walks for $\mu^{(\gamma)}$ and $\mu^{(0)}$.
As mentioned in Section \ref{ssec:limitsgamma}, 
the algorithm generalizes to the infinite settings of 
$\mu^{(0)}$ and $\mu^{(\gamma)}$; see \cite[Theorem 5.1]{BLPS01}
We start simple random walks at $w_1, \dots, w_K$. 
Each random walk receives an independent geometric time 
of parameter $\lambda = \gamma/(2d + \gamma)$. Ignoring the geometric
times, the algorithm realizes a sample from $\mu^{(0)}$.
When the random walks are killed at their respective 
geometric times, the algorithm realizes a sample from
$\mu^{(\gamma)}$, with killing corresponding to a jump
to $\varpi$. Let us call this type of coupling of
$\mu^{(0)}$ and $\mu^{(\gamma)}$ \emph{a standard coupling}.

It will be convenient to assume the following particular 
type of enumeration of vertices. Let $z_1, \dots, z_N$
be a list of all vertices in $\partial_\ext B(k)$, and let
$z_{N+1}, \dots, z_{N+M}$ be a list of all vertices of $B(k)$.
In addition we assume that the sequence $z_1, \dots, z_N$ 
has the property that for each $2 \le j \le N$
there exists $1 \le i(j) < j$ such that 
$z_{i(j)} \sim z_j$. We denote the spanning tree paths 
constructed in the coupling by $(\pi^{(0)}_{z_j})_{1 \le j \le N+M}$ 
and $(\pi^{(\gamma)}_{z_j})_{1 \le j \le N+M}$, respectively.
We will write $\bB(x,r)$ for the intersection with $\Z^d$
of the Euclidean ball of radius $r$ centred at $x$, and
$\bB(r)$ for $\bB(0,r)$.
The coupling is successful on the event when the following
conditions hold for some $m \ge 2k$:
\begin{itemize}
\item[(i)] for all $2 \le j \le N$, $\pi^{(0)}_{z_j}$ intersects 
$\pi^{(0)}_{z_{i(j)}}$ before leaving $\bB(m)$;
\item[(ii)] for all $2 \le j \le N$, $\pi^{(\gamma)}_{z_j}$
agrees with $\pi^{(0)}_{z_j}$ up to the first intersection of 
$\pi^{(0)}_{z_j}$ with $\pi^{(0)}_{z_{i(j)}}$;
\item[(iii)] for all $2 \le j \le N$, $\pi^{(\gamma)}_{z_{i(j)}}$ 
agrees with $\pi^{(0)}_{z_{i(j)}}$ up to the last exit of 
$\pi^{(0)}_{z_{i(j)}}$ from $\bB(m)$;
\item[(iv)] for all $N+1 \le j' \le N+M$, $\pi^{(\gamma)}_{z_{j'}}$
agrees with $\pi^{(0)}_{z_{j'}}$ up to the first intersection of
$\pi^{(0)}_{z_{j'}}$ with $\cup_{1 \le j \le N} \pi^{(0)}_{z_j}$.
\end{itemize}
Note that when (i)--(iii) hold, $\cup_{1 \le j \le N} \pi^{(0)}_{z_j}$
separates $B(k)$ from $\infty$, and hence condition
(iv) makes sense. 

We start by reducing the problem of verifying 
conditions (i)--(iii) to considering just two paths.
Let $(S(n))_{n \ge 0}$ and $(S^1(n))_{n \ge 0}$ be 
independent simple random walks started at the origin
and a neighbour of the origin, respectively, and let
$T$ and $T^1$ be independent $\mathrm{Geom}(\lambda)$
and random variables, independent of the walks. 
Sometimes it will be convenient to allow $T^1$ to be
$\mathrm{Geom}(\lambda_1)$ for some $\lambda_1$. We will
write $\Pr_{\lambda,\lambda_1}(\cdot)$ for the path space 
measure of the above coupling. Let
\eqnsplst
{ \xi^1
  &= \inf \{ j \ge 0 : S^1(j) \in \LE(S[0,\infty)) \} \\
  \xi^{1,\lambda}
  &= \inf \{ j \ge 0 : S^1(j) \in \LE(S[0,T)) \}. }

\begin{lemma}
\label{lem:2paths}
For any fixed $2 \le j \le N$, the joint distribution of 
\eqnst
{ (\pi^{(0)}_{z_{i(j)}}, \pi^{(\gamma)}_{z_{i(j)}}, 
     \pi^{(0)}_{z_j}, \pi^{(\gamma)}_{z_j}) }
is the same as the joint distribution of 
\eqnst
{ (\LE(S[0,\infty)), \LE(S[0,T]), 
     \LE(S^1[0,\xi^1]), \LE(S^1[0,\xi^{1,\lambda} \wedge T^1])), }
up to a shift by $z_{i(j)}$, and a rotation.
\end{lemma}

\begin{proof}
It will be useful to realize the coupling via cycle popping
from stacks of arrows, as in \cite{Wilson}.
For each $x \in \Z^3$, consider an infinite i.i.d.~stack of 
arrows pointing to random neighbours of $x$, with the stacks
also being independent. In addition, attach to each arrow, 
independently, a red marker with probability 
$\lambda$. If we ignore the markers, then we have 
the usual cycle popping in $\Z^3$. Now suppose that 
each arrow that has a marker is replaced by an arrow
pointing to $\varpi$. Then cycle popping realizes 
the Wilson algorithm for $\mu^{(\gamma)}$. 

Consider cycle popping starting at $z_{i(j)}$, that is,
follow the arrows, popping each cycle found. This uncovers
the path $\pi^{(0)}_{z_{i(j)}}$, and if the red markers are considered,
the path $\pi^{(\gamma)}_{z_{i(j)}}$. The paths constructed have
the joint distribution of $(\LE(S[0,\infty)), \LE(S[0,T]))$,
appropriately shifted. Now continue with cycle popping 
starting from $z_j$, which uncovers the paths $\pi^{(0)}_{z_j}$ and
$\pi^{(\gamma)}_{z_j}$. Interpreting cycle popping as a random
walk, we see that the conditional distribution of the new paths,
given the paths already constructed is the same as the 
conditional distribution of 
$(\LE(S^1[0,\xi^1]), \LE(S^1[0,\xi^1 \wedge T]))$, given
$(\LE(S[0,\infty)), \LE(S[0,T^1]))$ (appropriately shifted).
\end{proof}

Let us write 
\eqnsplst
{ \tau_m
  &= \inf \{ j \ge 0 : S(j) \not\in \bB(m) \} \\
  \tau^1_m
  &= \inf \{ j \ge 0 : S^1(j) \not\in \bB(m) \}. }
Consider the event that the following occur:
\begin{itemize}
\item[(i')] $\tau^1_m < \xi^1$, that is, $S^1$ hits 
$\LE(S[0,\infty))$ before exiting the ball $\bB(m)$;
\item[(ii')] $T^1 \ge \tau^1_m$, that is, $S^1$ is not
killed before exiting $\bB(m)$;
\item[(iii')] $\LE(S[0,T])$ agrees with $\LE(S[0,\infty))$
up to the last exit of $\LE(S[0,\infty))$ from $\bB(m)$.
\end{itemize}

Then we have the following corollary to Lemma \ref{lem:2paths}.

\begin{corollary}
\label{cor:reduce}
The probability that (i)--(iii) do not all occur is at most 
$N-1$ times the probabiliity that (i')--(iii') do not all
occur.
\end{corollary}

Finally, the probability that (iv) does not occur can be 
controlled by $M$ times the probability that 
\begin{itemize}
\item[(iv')] $T > \tau_{B(k)}$. 
\end{itemize}
does not occur.

When $d = 2$, $\LE(S[0,\infty))$ is not defined, and we 
cannot use a standard coupling in the infinite setting. 
Using a standard coupling in a large finite ball $\bB(N)$
is also problematic, due to recurrence, if $N$ is
extremely large with respect to $\lambda$. Indeed, most of the
work in the case $d = 2$ will be to show the existence
of a suitable coupling between the paths $\pi^{(0)}_0$ and  
$\pi^{(\gamma)}_0$. Once this is done, we control the rest
of the paths with a Beurling estimate.

Throughout we write $C$, $c$, etc.~to denote 
constants whose value may change from line to line.

\section{Rate of convergence estimate for $d = 3$}

Let $S=\{ S(j) \}_{j=0}^\infty$ be a simple random walk 
started at the origin. Let 
$\{ \hat{S}(j) \}_{j=0}^\infty$ be the loop-erasure 
of $S$. We write
$\tau_N = \inf \{ j \ge 0 : S(j) \not\in \bB(N) \}$,
and 
$\xi_m = \inf \{ j \ge 0 : S(j) \in \bB(m) \}$.

We say that $k \ge 0$ is a \emph{cut time} for the
random walk $S$, if $S[0,k] \cap S[k+1,\infty) = \es$.
It was shown in \cite{Lawler1996}, that there are 
constants $c_1, c_2, \zeta_3 > 0$ such that 
\eqnsplst
{ c_1 k^{-\zeta_3} 
  &\le \Pr [ S[0,k] \cap S[k+1,\infty) = \es ] \\
  &\le \Pr [ S[0,k] \cap S[k+1,2k] = \es ] \\
  &\le c_2 k^{-\zeta_3}, \quad k \ge 1, }  
where $\zeta_3$ is called the 
\emph{intersection exponent} in dimension $3$.
It was also shown in \cite{Lawler1996} that 
with probability bounded away from $0$,
there are at least $c k^{1-\zeta_3}$ cut times in
$[k,2k]$, and that with 
$R_k$ denoting the number of cut times in $[0,k]$,
$\log R_k / \log k \to 1 - \zeta_3$ with probability $1$.
Lawler's proof of the last result also gives an upper bound 
of $(\log n)^{-C}$ for the probability that there is no cut time in 
$[\tau_n,\tau_{n (\log n)^c} ]$. 
Essentially the same proof implies the lemma below;
see \cite[Corollary 4.12]{Lawler1996}.

\begin{lemma}
\label{lem:cuttime}
Assume $d = 3$. There exist constants $0 < \beta < 1$, 
$C < \infty$, such that for $m < n$ we have
\eqnspl{e:cuttimebound}
{ &\Pr [ \text{there is a cut time $k \in [\tau_m,\tau_n]$
    such that $S[k,\infty) \cap \bB(m) = \es$} ] \\
  &\qquad \ge 1 - C (m/n)^\beta. }
\end{lemma}
\qed

\begin{lemma}
\label{lem:LEnotreturn}
Assume $d = 3$. Let $m < n$. There exist constants
$C, \alpha < \infty$, such that for 
$0 < \lambda \le (m/n)^{\beta} (\alpha n^2 \log (n/m))^{-1}$ we have
\eqn{e:LEnotreturn}
{ \Pr [ \LE(S[0,T]) \cap \bB(m)= 
      \hat{S}[0,\infty) \cap \bB(m) ]
  \geq 1- C (m/n)^{\beta}, }
with $\beta$ as in Lemma \ref{lem:cuttime}.
\end{lemma}

\begin{proof}
Consider the event in \eqref{e:cuttimebound}. We claim that 
on this event, loop-erasure after time $\tau_n$ cannot
change the intersection of the path with $\bB(m)$.  Indeed,
since the path does not return to $\bB(m)$ after time 
$\tau_n$, the only way such change could occur if a loop
touching $\bB(m)$ is closed. Any such loop necessarily
had to start before the cut time, since after the cut time
$\bB(m)$ is not visited. However, by definition, the cut time
prevents such a loop to be closed.

Since $\Pr [ \tau_n > n^2 ] \le c_1 < 1$, 
there exist constants $c_2 > 0$, $C_2 < \infty$, such that 
for $x \ge 1$ we have 
$\Pr [ \tau_n > x n^2 ] \le C_2 \exp( - c_1 x )$. In particular,
$\Pr [ \tau_n > \alpha n^2 \log (n/m) ] \le C_2 (m/n)^{\alpha c_1}$.
Choosing $\alpha = \beta/c_1$ makes this bound 
$C_2 (m/n)^{\beta}$. Now consider 
\eqnsplst
{ \Pr [ T \le \alpha n^2 \log (n/m) ]
  = 1 - (1-\lambda)^{\alpha n^2 \log (n/m)} 
  \le \lambda \alpha n^2 \log (n/m)
  \le (m/n)^{\beta}. }
When $T > \alpha n^2 \log (n/m)$ and $\tau_n \le \alpha n^2 \log(n/m)$,
we have, $T > \tau_n$, and hence
the event in \eqref{e:cuttimebound} implies the event in 
\eqref{e:LEnotreturn}. The probability that any one of the required
events does not occur is at most $C (m/n)^{\beta}$, as required.
\end{proof}

Let 
\eqnst
{ F(\lambda,\lambda_1)
  = \Pr [ \LE(S[0,T]) \cap S^1[0,T^1] = \es ]. }
Write 
$\tau^1_n = \inf \{ k \ge 0 : S^1(k) \not\in \bB(n) \}$.

\begin{lemma}
\label{lem:S1nothit}
Assume $d = 3$. Let $m < n$,
$0 < \lambda \le (m/n)^\beta (\alpha n^2 \log(n/m))^{-1}$.
Then there exists constants $C, \delta < \infty$,
such that
\eqn{e:S1nothit}
{ \Pr_{\lambda,\lambda} [ \LE(S[0,T]) \cap S^1[0,\tau^1_m] = \es ]
  \le C m^{-1/3} (\delta \log m)^{1/2}. }
\end{lemma}

\begin{proof}
Let $\lambda_1 = m^{-2} (\delta \log m)^3$. Note that
$\lambda < \lambda_1$.
Using a large deviation bound for $\tau^1_m$
(see \cite[Lemma 1.5.1]{Lawler}) we have
\eqnsplst
{ \Pr_{\lambda_1} [ T^1 > \tau^1_m ]
  &\le \Pr_{\lambda_1} [ T^1 > m^2 / (\delta \log m)^2 ]
      + \Pr [ \tau^1_m \le m^2 / (\delta \log m)^2 ] \\
  &\le (1-\lambda_1)^{m^2/(\delta \log m)^2}
      + C_1 \exp(- \delta \log m) \\
  &\le \frac{C_2}{m^{\delta}}. }
Hence the probability in \eqref{e:S1nothit}
is at most
\eqnspl{e:F}
{ &\Pr_{\lambda,\lambda} [ \LE(S[0,T]) 
     \cap S^1[0,\tau^1_m \wedge T^1] = \es ] \\
  &\qquad \le \Pr_{\lambda,\lambda_1} [ \LE(S[0,T]) 
     \cap S^1[0,\tau^1_m \wedge T^1] = \es ] \\
  &\qquad \le \Pr_{\lambda,\lambda_1} [ \LE(S[0,T]) 
     \cap S^1[0,T^1] = \es ]
     + \Pr_{\lambda_1} [T^1 > \tau^1_m] \\
  &\qquad \le F(\lambda,\lambda_1) + C_2 m^{-\delta} \\ 
  &\qquad \le F(\lambda_1,\lambda_1) + C_2 m^{-\delta}. }
It was proved by Lawler \cite[Section 12.6]{Lawler1999},
that $F(\lambda_1,\lambda_1) \le C \lambda_1^{1/6}$.
(There the walk $S^1$ also starts at the origin,
however, it is straightforward to deduce the case
needed here). Choose $\delta > 1/3$. Then the 
right hand side of \eqref{e:F} is at most 
$C m^{-1/3} (\delta \log m)^{1/2}$.
\end{proof}

\begin{proof}[Proof of Theorem \ref{thm:main}; $d = 3$.]
We choose $n = m^{1+(3 \beta)^{-1}}$ in Lemmas \ref{lem:LEnotreturn} and
\ref{lem:S1nothit}, which makes the upper bounds in those
lemmas $m^{-1/3+o(1)} = \lambda^{\beta(2 + 7 \beta)^{-1}+o(1)}$. 

Let $z_1, \dots, z_N$ be an enumeration of the vertices
in $\partial_\ext B(k)$, and $z_{N+1}, \dots, z_{N+M}$
an enumeration of the vertices of $B(k)$. We assume 
that for each $2 \le j \le N$ there exists 
$1 \le i(j) < j$ such that $z_{i(j)} \sim z_j$.
Note that when the events in the statements of 
Lemmas \ref{lem:LEnotreturn} and \ref{lem:S1nothit} 
occur, then (i')--(iii') occur.
Hence the probability that there exists 
$2 \le j \le N$ for which (i)--(iii) do not
occur is at most $C k^2 \lambda^{\beta(2 + 7 \beta)^{-1}+o(1)}$.
Note that the union of the paths 
$\cup_{i=1}^N \pi^{(0)}_{z_i}$ disconnects $B(k)$ from 
$\partial \bB(m)$. Hence the random walks 
$S^i$, $i=N+1, \dots, N+M$ necessarily hit
the earlier paths. It follows that the probability 
that (iv) does not occur for some
$N+1 \le i \le N+M$ is at most
$M$ times the probability  
$\Pr [ T < \tau_m ]$. This can be bounded as
$C k^3 \lambda (\log k) k^2$. 
Therefore, Theorem \ref{thm:main} follows in the case 
$d = 3$. 
\end{proof}

\section{Rate of convergence estimate for $d = 2$}

The proof in the case $d=2$ follows a somewhat 
different outline, for two reasons. First, 
the loop-erasure of $S[0,\infty)$ cannot be defined
due to recurrence, so it has to be replaced by
the infinite loop-erased random walk 
\cite[Section 7.4]{Lawler}, leading
to a different coupling.  
Second, there are no global cut-times, so we will 
work with a finite volume analogue.

Here is the outline of the proof. We first couple
a suitable initial segment of $\LE(S[0,T])$
to an initial segment of the infinite LERW, and show 
that for suitable $m$, with high probability, 
$\LE(S[0,T]) \cap \bB(m)$ is determined by this
initial segment alone. This will give the required
coupling between the paths starting at $0$
for the measures $\mu^{(\gamma)}$ and $\mu^{(0)}$.
Then we use Wilson's method to generate the 
paths starting in $B(k) \cup \partial_\ext B(k)$, 
using the same random walks for $\mu^{(\gamma)}$ 
and $\mu^{(0)}$.
We show that with high probability the new paths
all stay inside $\bB(m)$ and are not killed
before their respective hitting times.

Let $\hat{S}^\lambda$, respectively $\hat{S}^N$, denote the 
loop-erasures of $S[0,T]$, respectively $S[0,\tau_N]$. 
Let $\Gamma_r$ denote the set of 
$r$-step self-avoiding paths starting at $0$.
We define the measures
$\hat{P}^\lambda = \hat{P}^\lambda_r$ and 
$\hat{P}^N = \hat{P}^N_r$ on $\Gamma_r$
by the formulas
\eqnst
{ \hat{P}^\lambda(\gamma)
  = \Pr \left[ \hat{S}^\lambda[0,r] = \gamma \right]
  \qquad \text{ and } \qquad  
  \hat{P}^N(\gamma)
  = \Pr \left[ \hat{S}^N[0,r] = \gamma \right]. }
Note that $\hat{P}^\lambda_r$ is not a probability 
measure, since $T < r$ has positive probability.
It was shown in \cite[Section 7.4]{Lawler} that
for every $r \ge 1$ the limit 
$\lim_{N \to \infty} \hat{P}^N_r =: \hat{P}_r$
exists, and hence $\hat{S}^N$ converges weakly 
to a limit $\hat{S}[0,\infty)$, called the 
infinite LERW.

Let $\xi_A = \inf \{ j \ge 1 : S(j) \in A \}$.
Below we write $\Pr^y[ \,\cdot\, ]$ for probability
under which $S$ starts at $y$, and the clock of
$T$ starts at $0$.

The following lemma can be proved similarly to
\cite[Proposition 7.3.1]{Lawler}.

\begin{lemma}
\label{lem:Laplwithkill}
If $\gamma_r = [\gamma(0),\dots,\gamma(r)] \in \Gamma_r$, 
$r \ge 1$, and $\gamma_{r-1} = [\gamma(0),\dots,\gamma(r-1)]$,
then 
\eqnsplst
{ \hat{P}^\lambda(\gamma_r)
  &= \hat{P}^\lambda(\gamma_{r-1})
    \Pr^{\gamma(r-1)} [ S(1) = \gamma(r),\, T \ge 1 
    \,|\, \xi_A > T ], }
where $A = \{ \gamma(0), \dots, \gamma(r-1) \}$. 
Also, 
\eqnst
{ \Pr [ \hat{S}^\lambda = \gamma_{r-1} ]
  = \hat{P}^\lambda(\gamma_{r-1}) 
    \Pr^{\gamma(r-1)} [ T = 0 \,|\, \xi_A > T ]. }
\end{lemma}

The lemma implies that given the first $r-1$ steps
of the loop-erased walk, the $r$-th step of the walk
will be to $y \sim \gamma(r-1)$, $y \not\in A$, 
with probability proportional to
\eqn{e:proportions}
{ \frac{1-\lambda}{4} \Pr^y [ \xi_A > T ], }
and that with probability proportional to $\lambda$
the walk has no $r$-th step.

Our first goal is to show that when $\lambda$ is small, 
a suitable initial segment of $\hat{S}^\lambda$ can be 
coupled to the corresponding initial segment of 
$\hat{S}$ with high probability. This will be
achieved by adapting the proof of 
\cite[Proposition 7.4.2]{Lawler}. We start by
identifying a set of ``good'' paths, where the
probabilities in \eqref{e:proportions} will behave
sufficiently regularly.

Let $\caT^N$ denote the wired uniform spanning tree
in $\bB(N)$ \cite{BLPS01}. Due to Wilson's algorithm,
the distribution of $\hat{S}^N$ is 
the same as the distribution of the path $\gamma^N$ 
connecting $0$ to the wired vertex in 
$\caT^N$. For $1 \le R < N$,
and $x \in \bB(R)$, let $\gamma^{x,N}$ 
denote the unique self-avoiding path in $\caT^N$ 
connecting $x$ to $0$. 
Let $\beta > 0$ be a parameter, whose
value will be fixed later.
We will measure the ``badness'' of paths
through the random variable
\eqnst
{ X^{(N)}_R
  := \# \{ x \in \bB(R) : 
    \gamma^{x,N} \subset \bB(R),\, 
    x \in \gamma^N,\, 
    \exists y \sim x\ 
    \Pr^y [ \xi_{\gamma^x} > \tau_{2R} ] \le R^{-\beta} \}. }

\begin{lemma}
\label{lem:comparable}
Assume $d = 2$. Suppose we have 
$\gamma = [\gamma(0),\dots,\gamma(\ell)]$,
$\gamma \subset \bB(R)$, and $y_1, y_2 \sim \gamma(\ell)$
such that $\Pr^{y_i} [ \xi_{\gamma} > \tau_{2R} ] > 0$,
$i=1,2$. Then 
\eqnst
{ \Pr^{y_1} [ \xi_\gamma > \tau_{2R} ] 
  \ge \frac{1}{256} \Pr^{y_2} [ \xi_\gamma > \tau_{2R} ]. }
\end{lemma}

\begin{proof}
It is sufficient to show that there exists a path of 
at most 4 steps from $y_1$ to $y_2$ that avoids $\gamma$. 
It follows from the planarity of the configuration that
such a path exists.
\end{proof}

\begin{lemma}
\label{lem:badness} 
There exists $C < \infty$, such that for all
$N > 2R$ we have 
$\Pr [ X^{(N)}_R \ge 1 ] \le C (\log R) R^{2-\beta}$.
\end{lemma}

\begin{proof}
We use Wilson's algorithm rooted at $0$. 
We write $\varpi$ for the wired vertex of the graph
$G^{(0)}_{\bB(N)}$ obtained from $\bB(N)$. First generate 
the path $\gamma^{x,N}$ by running a LERW from $x$ to $0$.
Next run a LERW from $\varpi$. Then
\eqnspl{e:Xnk}
{ \E [ X^{(N)}_R ] 
  &= \sum_{x \in \bB(R)} 
    \E \Big[ \Pr^\varpi [ \text{first hit $\gamma^{x,N}$ at $x$} ]
    I [ \gamma^{x,N} \subset \bB(R) ] \\
    &\qquad\qquad\qquad I [ \exists y \sim x : 
    0 < \Pr^y [ \xi_{\gamma^{x,N}} > \tau_{2R} ] \le R^{-\beta} ] 
    \Big]. }
Let $G^N(u,v)$ denote the Green function of random walk
on the wired graph $G^{(0)}_{\bB(N)}$, killed upon
hitting $0$, and let $G^{\gamma^{x,N},N}(u,v)$ denote the Green 
function of random walk on the same graph, 
killed upon hitting $\gamma^{x,N}$. 
Let $D_N = \text{degree of $\varpi$}$, and note that
$D_N \ge c N$. Using reversibility
of the random walk, we have
\eqnspl{e:estimbad}
{ \Pr^\varpi [ \text{first hit $\gamma^{x,N}$ at $x$} ]
  &= \frac{4}{D_N} G^{\gamma^{x,N},N}(\varpi,\varpi) 
    \Pr^x [ \text{no return to $\gamma^{x,N}$ before $\tau_N$} ] \\
  &= \frac{4}{D_N} G^{\gamma^{x,N},N} (\varpi,\varpi)
    \frac{1}{4} \sum_{y \sim x} \Pr^y [ \xi_{\gamma^{x,N}} > \tau_N ]. }
We have $G^{\gamma^{x,N},N} (\varpi,\varpi) \le G^N(\varpi,\varpi)$.
For $u$ in the interior boundary of $\bB(N)$, the probability 
that a random walk started at $u$ hits $0$ before exiting 
$\bB(N)$ is bounded below by $c/(N \log N)$, with $c$ 
independent of $u$ and $N$. This follows from the facts 
that there is probability at least $c/N$ for the walk 
to reach $\bB(N/2)$ before exiting $\bB(N)$, and there
is probability $c/\log N$ for it to reach $0$ from 
$\partial \bB(N/2)$ before exiting $\bB(N)$ 
(see \cite[Exercise 1.6.8]{Lawler}). 
This implies that $G^N(\varpi,\varpi) \le C N \log N$. 
Each term in the sum over $y$ is either $0$, 
or can be bounded by 
\eqnst
{ \Pr^y [ \xi_{\gamma^{x,N}} > \tau_{2R} ] \sup_{z \in \partial \bB(2R)}
     \Pr^z [ \xi_0 > \tau_N ]
  \le C R^{-\beta} \frac{\log R}{\log N}. }
due to Lemma \ref{lem:comparable}.
We obtain that the right hand side 
of \eqref{e:estimbad} is less than or equal to 
\eqnst
{ \frac{C}{D_N} G^N(\varpi,\varpi) \frac{\log R}{\log N} R^{-\beta}
  \le C (\log R) R^{-\beta}. }
Inserting this into \eqref{e:Xnk}, we obtain the statement
of the lemma.
\end{proof}

We will require that $\beta > 2$.
We introduce the length scale $n$ of the form 
$n = \lambda^{-\rho}$, where the exponent 
$\rho$ will be chosen at the end of the proof.
Its role will be to allow us replace the killing time
$T$ by $\tau_n$ in such a way that with high
probability $T$ occurs later than $\tau_n$,
but not much later (up to a power).
We will chose $R$ of the form $R = \lambda^{-\rho'}$,
with $\rho' < \rho < 1/2$, so that $R \ll n \ll 1/\sqrt{\lambda}$,
and $\rho'$ will also be 
chosen at the end of the proof.
In the rest of this section, 
$c, c', C, C'$ may depend on the exponents 
$\rho$, $\rho'$, etc., but
do not depend on $\lambda$. 

\begin{proposition}
\label{prop:convLERW}
Suppose that $1 - 2 \rho > (1+\beta) \rho'$ and
$2 \rho' < \rho$.
There exists a subset $\caP \subset \Gamma_R$ such that
$\hat{P}(\caP) \ge 1 - C (\log R) R^{2-\beta}$ 
and for all $\gamma \in \caP$ we have
\eqnst
{ \hat{P}^\lambda(\gamma)
  = \hat{P}(\gamma) 
    \left( 1 + O \left( \lambda n^2 R^{1+\beta} 
    / \log [ \lambda^{-1} ] \right)
    + O \left( \frac{R^2}{n} \log \frac{n}{R} \right) \right). }
\end{proposition}

\begin{proof}
The proof goes by adapting the argument of
\cite[Proposition 7.4.2]{Lawler}. The main 
difference from that proof is that ``trapping'' can occur,
and convergence does not hold uniformly
over all paths. 

Let $\caP$ be the set of paths 
$\gamma = [\gamma(0),\dots,\gamma(R)]$ such that for all
$1 \le \ell \le R$ we have
$\Pr^{\gamma(\ell)} [ \xi_{\gamma_\ell} > \tau_{2R} ] \ge R^{-\beta}$, 
with $\gamma_\ell := [\gamma(0),\dots,\gamma(\ell)]$
and for all $y \sim \gamma(\ell-1)$, $y \not\in \gamma_\ell$
we have either $\Pr^y [ \xi_{\gamma_\ell} > \tau_{2R} ] = 0$ 
or $\Pr^y [ \xi_{\gamma_\ell} > \tau_{2R} ] \ge R^{-\beta}$.
Then Lemma \ref{lem:badness} and 
Lemma \ref{lem:comparable} show that 
\eqnst
{ \hat{P}(\caP)
  = \lim_{N \to \infty} \hat{P}^N(\caP) 
  \ge 1 - \limsup_{N \to \infty} \Pr [ X^{(N)}_R \ge 1 ] 
  \ge 1 - C (\log R) R^{2-\beta}. }

For $1 \le j \le R$, Lemma \ref{lem:Laplwithkill} implies
\eqn{e:Laplformula}
{ \hat{P}^\lambda(\gamma_j)
  = \hat{P}^\lambda(\gamma_{j-1})
    \frac{\frac{1-\lambda}{4} \Pr^{\gamma(j)} [ \xi_A > T ]}%
    {\lambda + \frac{1-\lambda}{4} 
    \sum_{y \not\in A,\, y \sim \gamma(j-1)} \Pr^y [ \xi_A > T ]}, }
where $A = A_j = \{ \gamma(0), \dots, \gamma(j-1) \}$. 
The corresponding formula for $\hat{P}^n$
is \cite[Eqn.~(7.3)]{Lawler}:
\eqnst
{ \hat{P}^n(\gamma_j)
  = \hat{P}^n(\gamma_{j-1})
    \frac{\Pr^{\gamma(j)}[\xi_A > \tau_n]}%
    {\sum_{y \not\in A,\, y \sim \gamma(j-1)} \Pr^y [ \xi_A > \tau_n]}. }

Assuming $\gamma \in \caP$, and $1 \le j \le R$,
we start by relating $\Pr^y [ \xi_A > T ]$ 
to $\Pr^y [ \xi_A > \tau_n ]$. We can write 
\eqnspl{e:Pdecomp1}
{ \Pr^y [ \xi_A > T ] 
  &= \Pr^y [ \xi_A > T,\, T < \tau_n ] \\
  &\qquad + \Pr^y [ \xi_A > \tau_n,\, T \ge \tau_n ]
      \Pr^y [ \xi_A > T | \xi_A > \tau_n,\, T \ge \tau_n ]. }
The first term on the right hand side of \eqref{e:Pdecomp1}
will be an error term, that we estimate as follows.
Let $\alpha > 0$ be a large parameter, that we will choose
in the course of the proof. 
\eqnspl{e:smallT}
{ \Pr^y [ \xi_A > T,\, T < \tau_n ]
  &\le \Pr^y [ \tau_n > \alpha (\log n) n^2 ] 
      + \Pr^y [ T < \alpha (\log n) n^2 ] \\
  &\le \exp( - c \alpha \log n )
      + \lambda \alpha (\log n) n^2 \\
  &\le C \lambda (\log n) n^2. }
The last step is justified, if we take 
$\alpha > (\frac{1}{\rho} - 2)/c$.

Let us consider the second term on the right hand
side of \eqref{e:Pdecomp1}. Using the strong Markov
property, and the memoryless property of $T$, we can write
\eqnspl{e:Pdecomp2}
{ &\Pr^y [ \xi_A > T | \xi_A > \tau_n,\, T \ge \tau_n ] \\
  &\qquad = \sum_{z \in \partial \bB(n)} 
     \Pr^y [ S(\tau_n) = z | \xi_A > \tau_n,\, T \ge \tau_n ]
     \Pr^z [ \xi_A > T ]. }
Let $H_n(z) := \Pr^0 [ S(\tau_n) = z ]$. The key step of the 
proof is to show that the first factor of the summand 
on the right hand side of \eqref{e:Pdecomp2} is 
essentially independent of $y$, and equals
$H_n(z) ( 1 + O(\frac{R}{n} \log \frac{n}{R} ) 
+ O (\lambda (\log n) n^2) )$.
Using the strong Markov property, and the memoryless 
property of $T$, we can write
\eqnspl{e:Pdecomp3}
{ &\Pr^y [ \xi_A > \tau_n,\, T \ge \tau_n,\, S(\tau_n) = z ] \\
  &\qquad = \sum_{w \in \partial \bB(2R)} 
     \Pr^y [ \xi_A > \tau_{2R},\, 
     T \ge \tau_{2R},\, S(\tau_{2R}) = w ] \\
  &\qquad\qquad\qquad\quad \times \Pr^w [ \xi_A > \tau_n,\, 
     T \ge \tau_n,\, S(\tau_n) = z ]. }
By \cite[Eqn.~(2.10)]{Lawler}, we have
\eqn{e:Lawler(2.10)}
{ \Pr^w [ \xi_A > \tau_n,\, S(\tau_n) = z ]
  = \Pr^w [ \xi_A > \tau_n ] H_n(z)
    \left( 1 + O \left( \frac{R}{n} 
    \log \frac{n}{R} \right) \right). }
Hence we want to estimate the effect of omitting the event
$T \ge \tau_n$ from the second probability on the
right hand side of \eqref{e:Pdecomp3}.
By a standard estimate \cite[Exercise 1.6.8]{Lawler},
we have 
$\Pr^w [ \xi_A > \tau_n ] 
\ge \Pr^w [ \xi_{\bB(R)} > \tau_n ] 
\ge c [\log \frac{n}{R}]^{-1}$.
We also have $H_n(z) \ge c n^{-1}$, by 
\cite[Lemma 1.7.4]{Lawler}. This implies that 
\eqn{e:Pdecomp4}
{ \Pr^w [ \xi_A > \tau_n,\, S(\tau_n) = z ] 
  \ge c \left( n \log \frac{n}{R} \right)^{-1}. }
We are now ready to estimate
\eqnspl{e:Pdecomp5}
{ &\Pr^w [ \xi_A > \tau_n,\, T < \tau_n,\, S(\tau_n) = z ] \\
  &\qquad \le \Pr^w [ \tau_n > \alpha (\log n) n^2 ] 
      + \Pr^w [ \xi_A > \tau_n,\, 
      T < \alpha (\log n) n^2,\, S(\tau_n) = z ] \\
  &\qquad \le \exp( - c \alpha \log n )
      + \lambda \alpha (\log n) n^2 
      \Pr^w [ \xi_A > \tau_n,\, S(\tau_n) = z ] \\
  &\qquad \le \Pr^w [ \xi_A > \tau_n,\, S(\tau_n) = z ]
      O( \lambda (\log n) n^2 ). }
Here we have used the lower bound \eqref{e:Pdecomp4}
and we require that $\alpha$ satisfy
$\alpha > (\frac{1}{\rho} - 1)/c$.
Putting the estimates \eqref{e:Lawler(2.10)} and
\eqref{e:Pdecomp5} together we have
\eqnsplst
{ \Pr^w [ \xi_A > \tau_n,\, T \ge \tau_n,\, S(\tau_n) = z ]
  &= \Pr^w [ \xi_A > \tau_n ] H_n(z) \\
  &\qquad\qquad \times 
    \left( 1 + O \left( \frac{R}{n} \log \frac{n}{R} \right) 
    + O ( \lambda (\log n) n^2 ) \right). }
Putting this back into \eqref{e:Pdecomp3}, we get
the key estimate:
\eqn{e:key}
{ \Pr^y [ S(\tau_n) = z | \xi_A > \tau_n,\, T \ge \tau_n ] 
  = H_n(z) \left( 1 + O \left( \frac{R}{n} 
    \log \frac{n}{R} \right) 
    + O (\lambda (\log n) n^2) \right). }
Inserting \eqref{e:key} into \eqref{e:Pdecomp2} 
and summing over $z$, we get
\eqn{e:Ddef}
{ \Pr^y [ \xi_A > T | \xi_A > \tau_n,\, T \ge \tau_n ]
  = D(\lambda,n,A) 
    + O \left( \frac{R}{n} \log \frac{n}{R} \right)
    + O(\lambda (\log n) n^2), }
where 
$D(\lambda,n,A) 
= \sum_{z \in \partial \bB(n)} H_n(z) \Pr^z [ \xi_A > T ]$.

Now we are ready to start analyzing the expression
\eqref{e:Laplformula}. 
Consider $y \sim \gamma(j-1)$, $y \not\in A$,
such that $\Pr^y [ \xi_A > \tau_{2R} ] > 0$.
The main contribution to $\Pr^y [ \xi_A > T ]$ will be
$\Pr^y [ \xi_A > T,\, \xi_A > \tau_n ]$.
An estimate very similar to \eqref{e:smallT}
yields
\eqnspl{e:tooshort}
{ \Pr^y [ \xi_A > T,\, \xi_A < \tau_n ]
  \le C \lambda (\log n) n^2. }
A lower bound for the main term is as follows.
\eqnspl{e:main}
{ &\Pr^y [ \xi_A > T,\, \xi_A > \tau_n ] \\
  &\qquad \ge \Pr^y [ \xi_A > \tau_{2R} ]
       \Pr^y [ \xi_A > \tau_{1/\sqrt{\lambda}} | \xi_A > \tau_{2R} ]
       \Pr^y [ \xi_A > T | \xi_A > \tau_{1/\sqrt{\lambda}} ]. }
The first factor on the right hand side of \eqref{e:main} 
is at least $c R^{-\beta}$, due to $\gamma \in \caP$. 
The second factor is at least the minimum over $z$ 
of the probability that random walk started at 
$z \in \partial \bB(2R)$ will exit 
$\bB(1/\sqrt{\lambda})$ before hitting $\bB(R)$. 
This is at least 
$c \log (2R/R) / \log [ (\sqrt{\lambda} R)^{-1} ] 
> c' / \log [ \lambda^{-1} ]$. 
We show that the third factor is 
bounded away from $0$. Due to the invariance principle,
with probability bounded away from $0$, a random walk started
on the boundary of $\bB(1/\sqrt{\lambda})$ will take
at least $1/\lambda$ steps before hitting 
$\bB(1/(2 \sqrt{\lambda}))$, and hence before
hitting $A$. On this event, the conditional probability of
$\xi_A > T$ is at least 
$P [ T \le \lambda^{-1}] \approx 1 - e^{-1}$. 
Putting the estimates together, we get
\eqn{e:main2}
{ \Pr^y [ \xi_A > T,\, \xi_A > \tau_n ]
  \ge c R^{-\beta} / \log [ \lambda^{-1} ]. }
The choice $\lambda = n^{-\rho}$ and the condition on 
$\rho$ and $\rho'$ ensure that 
$C \lambda (\log n) n^2$ is of smaller 
order than the right hand side of \eqref{e:main2}.
Putting \eqref{e:tooshort} and \eqref{e:main2}
together, we obtain 
\eqn{e:main3}
{ \Pr^y [ \xi_A > T ] 
  \ge c R^{-\beta} / \log [ \lambda^{-1} ]. }
Arguments similar to what led to \eqref{e:main3}, also
yield the simple lower bound 
\eqn{e:Dbound}
{ D(\lambda,n,A) 
  \ge c. }
We return to \eqref{e:Pdecomp1}. The estimates 
\eqref{e:Ddef}, \eqref{e:tooshort}, \eqref{e:main3}
and \eqref{e:Dbound} imply
\eqnspl{e:error1}
{ \Pr^y [ \xi_A > T ]
  &= O(\lambda (\log n) n^2) 
    + \left( \Pr^y [ \xi_A > \tau_n ] 
    + O(\lambda (\log n) n^2) \right) \\
  &\qquad\qquad\qquad\quad 
    \times \left( D(\lambda,n,A_j) 
    + O \left( \frac{R}{n} \log \frac{n}{R} \right)
    + O( \lambda (\log n) n^2 ) \right) \\
  &= \Pr^y [ \xi_A > \tau_n ] \left( D(\lambda,n,A_j) 
    + O \left( \frac{R}{n} \log \frac{n}{R} \right)
    + O \left( \lambda n^2 R^\beta / \log [ \lambda^{-1} ] \right) 
    \right). }
Consider now the case when $y \sim \gamma(j-1)$, 
$y \not\in A_j$, such that $\Pr^y [ \xi_A > \tau_{2R} ] = 0$. 
In this case we have
\eqn{e:error2}
{ \Pr^y [ \xi_A > T ] 
  \le \Pr^y [ T < \tau_{2R} ]
  \le C \lambda \alpha (\log n) n^2. }
Putting \eqref{e:error1} and \eqref{e:error2} into 
\eqref{e:Laplformula}, we get
\eqn{e:condj}
{ \hat{P}^\lambda (\gamma_j) 
  = \hat{P}^\lambda (\gamma_{j-1})
    \frac{\hat{P}^n (\gamma_{j})}{\hat{P}^n(\gamma_{j-1})}
    \left( 1 + O(\lambda n^2 R^\beta / \log [ \lambda^{-1} ] ) 
    + O \left( \frac{R}{n} \log \frac{n}{R} \right) \right). }
Iterating for $j = 1, \dots, R$ we get 
\eqnst
{ \hat{P}^\lambda(\gamma)
  = \hat{P}^n (\gamma) 
    \left( 1 + O(\lambda n^2 R^{1+\beta} / \log [ \lambda^{-1} ] )
    + O \left( \frac{R^2}{n} \log \frac{n}{R} \right) \right). }
The proposition follows, since
due to \cite[Proposition 7.4.2]{Lawler}, we have 
\eqnst
{ \hat{P}^n(\gamma) 
  = \hat{P}(\gamma) 
    \left( 1 + O(\frac{R^2}{n} \log \frac{n}{R} ) \right). }
\end{proof}

Our next step is to get an estimate on the probability that 
$\hat{S} \cap \bB(m)$ differs from
$\LE(S[0,T]) \cap \bB(m)$ for suitable $m$.
We will select $m$ of the form $m = \lambda^{-\rho''}$, and
require $\rho'' < \rho'$, so that $m \ll R$.
We will need the discrete Beurling estimate, stated below.
This was first proved by Kesten \cite{Kesten}; see 
\cite{LawLim} for a version more similar to what 
will be used here.

\begin{theorem}[Beurling estimate, \cite{LawLim}]
\label{thm:Beurling}
Suppose $m < N$, and $A \subset \Z^2$ contains a path 
from $\bB(m)$ to $\partial \bB(N)$. There exists a constant
$C < \infty$ such that for any $z \in \partial \bB(m)$ we have
\eqnst
{ \Pr^z [ \tau_N < \xi_A ] 
  \le C (m/N)^{1/2}. }
\end{theorem}

The next proposition estimates the probability that the 
loop-erasure inside a ball $\bB(m)$ is affected
after the loop-erased path has reached distance
$R' > m$. Later on we are going to take $R' = (1/2) \sqrt{R}$
with $R$ as in Proposition \ref{prop:convLERW}.
We define 
\eqnsplst
{ \hat{\tau}^\lambda_{R'} 
  &= \inf \{ j \ge 0 : \hat{S}^\lambda(j) \in \partial \bB(R') \} \\
  \hat{\tau}_{R'} 
  &= \inf \{ j \ge 0 : \hat{S}(j) \in \partial \bB(R') \}. }

\begin{proposition}
\label{prop:2Dnotreturn}
Assume $d = 2$. Let $16m < R'$. 
Suppose that $\lambda \le (R')^{-4}$. 
There exists a constant $C < \infty$ such that
\eqnspl{e:2Dnotreturn1}
{ \Pr [ \text{$\hat{S}^\lambda (j) \not\in \bB(m)$ 
    for $j \ge \hat{\tau}^\lambda_{R'}$} ] 
  \ge 1 - C (m/R') \log(R'/m). }
Likewise, we have
\eqnspl{e:2Dnotreturn2}
{ \Pr [ \text{$\hat{S}(j) \not\in \bB(m)$ 
    for $j \ge \hat{\tau}_{R'}$} ] 
  \ge 1 - C (m/k') \log(k'/m). }
\end{proposition}

\begin{proof}
We estimate the probability that 
after the loop-erasure of $S[0,T]$ has reached 
$\partial \bB(R')$, the walk $S$ revisits 
$\bB(m)$, before time $T$.
Condition on the set 
$A = \hat{S}^\lambda[0,\hat{\tau}^\lambda_{R'}-1]$,
and let $x = \hat{S}^\lambda(\hat{\tau}^\lambda_{R'})$. 
Let 
\eqnst
{ \rho 
  = \max \{ j < T : S(j) = S(\hat{\tau}^\lambda_{R'}-1) \}. }
The law of $S[\rho+1,T]$ is that of a random walk 
started at $x$ and conditioned on $\xi_A > T$. 
We show that
\eqn{e:2Dreturn}
{ \Pr^x [ \xi_{\bB(m)} \le T | \xi_A > T ] 
  \le C_1 (m/R') \log(R'/m), }
which implies the claim of the proposition.
The walk first has to exit $\bB(x,R'/4)$ without
hitting $A$. Then it has to cross $\bB(R'/4) \setminus \bB(m)$ 
without hitting $A$, in order to visit $\bB(m)$. 
Both of these have to occur before time $T$.
Following the visit to $\bB(m)$, the walk has to 
avoid $A$ until time $T$.
Since $1/\sqrt{\lambda}$ is of larger order than $R'$,
this means that the walk will essentially have to 
cross $\bB(R'/4) \setminus \bB(m)$ again without hitting $A$,
and stay away from $A$ after that.
Since $A$ contains a path from $\bB(m)$ to $\partial \bB(R'/4)$,
the Beurling estimate Theorem \ref{thm:Beurling} 
can be used to bound the probability
that $A$ is not hit during the crossings. This will yield
a bound $C_2 (m/R')^{1/2} (m/R')^{1/2} \log(R'/m)$, where
the $\log$-factor arises due to a technicality. 
The probability of not hitting $A$ during the final stretch
will yield a factor $C/\log (1/\sqrt{\lambda}R')$, 
that will be used to cancel the 
effect of the conditioning.

We first get a lower bound for the probability of the
conditioning in \eqref{e:2Dreturn}.
Let $c_1$ be a constant such that 
the set $\{ y \in \partial \bB(x,R'/4) : |y| > (1+c_1)R' \}$
contains at least a fraction $c_2 > 0$ of $\partial \bB(x,R'/4)$. 
Let us write $\tau_{x,R'}$ for $\tau_{\bB(x,R'/4)}$.
We have
\eqnspl{e:exitk'}
{ &\Pr^x [ \xi_A > \tau_{x,R'},\, T \ge \tau_{x,R'},\, 
      |S(\tau_{x,R'})| > (1+c_1)R' ] \\
  &\qquad\qquad \ge \Pr^x [ \xi_A > \tau_{x,R'},\, 
      |S(\tau_{x,R'})| > (1 + c_1)R' ] \\
  &\qquad\qquad\quad - \Pr^x [ \tau_{x,R'} > \alpha (\log R') (R')^2 ] \\
  &\qquad\qquad\quad - \Pr^x [ T < \alpha (\log R') (R')^2 ]. }
We claim that the subtracted terms in \eqref{e:exitk'}
are of lower order than the first term.
The first term is at least the probability 
that $S$ exists $\bB(x,R'/4)$ without returning
to $\bB(R')$, and reaches a distance of order
$R'$ from $\bB(R')$. This probability is at least $c/R'$
\cite[Exercise 1.6.8]{Lawler}. 
The second term is at most $\exp( -c \alpha \log R')$,
which for $\alpha$ large enough is of smaller order
than the first term. The third term is 
$O(\lambda (\log R') (R')^2) = o(1/R')$, by the
condition on $\lambda$.
It is intuitive that given the event 
$\xi_A > \tau_{x,R'}$, we have 
$|S(\tau_{x,R'})| > (1 + c_1)R'$ with 
conditional probability bounded away from $0$.
A proof of this can be found in 
\cite[Proposition 3.5]{Masson}.
Hence we have
\eqn{e:exitk'2}
{ \Pr^x [ \xi_A > \tau_{x,R'},\, T \ge \tau_{x,R'},\, 
      |S(\tau_{x,R'})| > (1 + c_1)R' ]
  \ge c \Pr^x [ \xi_A > \tau_{x,R'}\,  ]. } 
The walk having reached distance
$(1 + c_1)R'$, the probability that the walk
will avoid $A$ until time $T$ is at least 
the probability that
(i) it exits $\bB(2/\sqrt{\lambda})$
without hitting $A$; and
(ii) it takes at least $1/\lambda$ steps 
before entering $\bB(1/\sqrt{\lambda})$;
and (iii) $T \le 1/\lambda$.
The probability of (i) is
at least $c \log(2(\sqrt{\lambda} R')^{-1})$. 
The probability of (ii) is at least $c$, due to
the invariance principle. The probability of
(iii) is bounded away from $0$.
This gives us
\eqn{e:separated}
{ \Pr^x [ \xi_A > T ]
  \ge c \Pr^x [ \xi_A > \tau_{x,R'} ]
      \log (2 (\sqrt{\lambda} R')^{-1}). } 

We now come to deriving an upper bound for
$\Pr^x [ \xi_{\bB(m)} \le T,\, \xi_A > T ]$.
First consider the walk up to its first 
exit from $\bB(x,R'/4)$. Then we see that the 
event
$\{ \xi_A > \tau_{x,R'},\, T \ge \tau_{x,R'} \}$ 
has to occur. We neglect the requirement that
$T \ge \tau_{x,R'}$, and use 
$\Pr^x [ \xi_A > \tau_{x,R'} ]$ as an upper bound.
We now consider the walk starting on 
$\partial \bB(x,R'/4)$.
We need to make precise the estimates on the
probability of crossing $\bB(R'/4) \setminus \bB(m)$.
For the first crossing, we can again neglect
the event $T \ge \xi_{\bB(m)}$.
We introduce the notation 
\eqnsplst
{ \xi_1 
  &= \inf \{ j \ge 0 : S(j) \in \partial \bB(R'/4) \} \\
  \xi_2
  &= \inf \{ j \ge \xi_1 : S(j) \in \bB(m) \} \\
  \bar{\xi}_1
  &= \sup \{ \xi_1 \le j \le \xi_2 : 
     S(j) \in \partial \bB(R'/4) \} \\
  Y 
  &= S(\bar{\xi}_1) \\
  Z
  &= S(\xi_2) \\
  \bar{\xi}'_1
  &= \sup \{ \xi_1 \le j \le \xi_2 : 
     S(j) \in \partial \bB(R'/8) \} \\
  \bar{\xi}'_2  
  &= \sup \{ \xi_1 \le j \le \xi_2 :
     S(j) \in \partial \bB(2m) \} \\
  Z'
  &= S(\bar{\xi}'_2). }
Fix $y \in \partial \bB(R'/4)$, $z \in \bB(m)$, and
$z' \in \partial \bB(2m)$.
Let $\tilde{S}$ denote a random walk started at $z$.
We will use tildes for stopping times corresponding
to $\tilde{S}$. Conditioned on the event 
$\{ \xi_2 < \infty,\, Y = y,\, Z = z,\, Z' = z' \}$,
the time reversal of $S[\bar{\xi}'_1, \bar{\xi}'_2]$ 
has the same law as 
$\tilde{S}[\tilde{\tau}_{2m},\tilde{\tau}_{R'/8}]$ 
conditioned on the event 
\eqnst
{ \tilde{E}_y 
  = \{ \tilde{S}(\tilde{\tau}_{2m}) = z',\, 
    \tilde{\tau}_{R'/4} < \tilde{\xi}_{\bB(m)},\, 
    \tilde{S}(\tilde{\tau}_{R'/4}) = y \}. } 
We have 
\eqnst
{ \Pr^{z'} [ \tilde{\tau}_{R'/4} < \tilde{\xi}_{\bB(m)} ] 
  \ge \frac{C_1}{\log(R'/m)}. }
Due to the Harnack principle 
\cite[Theorem 1.7.6]{Lawler}, conditioning on 
$\tilde{S}(\tilde{\tau}_{R'/4}) = y$ affects the 
probability of 
$\{ \tilde{S}[\tilde{\tau}_{2m},\tilde{\tau}_{R'/8}] 
\cap A = \es \}$ by a factor that is bounded away from 
$0$ and $\infty$.
Hence it follows that 
\eqnspl{e:inwards}
{ &\Pr [ \xi_2 < \infty,\, S[\xi_1,\xi_2] \cap A = \es ] \\
  &\qquad \le \sup_{z',y} \Pr^{z'} 
    [ \tilde{S}[0,\tilde{\tau}_{R'/8}] 
    \cap A = \es | \tilde{E}_y ] \\
  &\qquad \le C_1 \log(R'/m) \sup_{z'} \Pr^{z'} 
    [ \tilde{S}[0,\tilde{\tau}_{R'/8}] 
    \cap A = \es | \tilde{S}(\tilde{\tau}_{R'/4}) = y ] \\
  &\qquad \le C_2 \log(R'/m) \sup_{z'} \Pr^{z'} 
    [ \tilde{S}[0,\tilde{\tau}_{R'/8}] 
    \cap A = \es ] \\
  &\qquad \le C_3 (m/R')^{1/2} \log(R'/m). }
In the last step, we use Theorem \ref{thm:Beurling}.

For the other crossing of $\bB(R'/4) \setminus \bB(m)$
we apply Theorem \ref{thm:Beurling} directly.
Due to the memoryless property of $T$, we may assume
that the clock of $T$ is starting at time $\xi_2$.
Let
\eqnst
{ \tau_1 
  = \inf \{ j \ge \xi_2 : S(j) \not\in \bB(R'/4) \}. }
Then 
\eqnspl{e:outwards}
{ \Pr [ S[\xi_2,\tau_1] \cap A = \es | \xi_2 < \infty ]
  &\le \sup_{z \in \partial \bB(m)} 
      \Pr^z [ S[0,\tau_{R'/4}] \cap A = \es ] \\
  &\le C_4 (m/R')^{1/2}. }
The probability of $T < \tau_1$ can be estimated 
similarly to \eqref{e:exitk'}, and is of
smaller order than the right hand side 
of \eqref{e:outwards}. Therefore we also have
\eqn{e:outwards2}
{ \Pr [ S[\xi_2,\tau_1] \cap A = \es,\, T \ge \tau_1 | 
      \xi_2 < \infty,\, T \ge \xi_2 ]
  \le C_4 (m/R')^{1/2}. }

We finally bound the probability that $A$ is not hit
between $\tau_1$ and $T$, given that $T \ge \tau_1$. 
Consider $n' = \lambda^{-1/2+\eps}$. We have
$\Pr [ T \le \tau_{n'} | T \ge \tau_1 ] 
= O(\lambda (\log n') (n')^2 )
= o(\lambda^\eps)$. 
Hence we are going to consider the probability that 
the walk avoids $A$ up to time $\tau_{n'}$. Let 
\eqnst
{ \tau_2
  = \inf \{ j \ge \tau_1 : 
    S(j) \not\in \bB(2R') \}. }
Define inductively the sequence of stopping times 
\eqnsplst
{ \rho_1
  &= \inf \{ j : \tau_2 \le j \le \tau_{n'},\, 
    S(j) \in \bB(R'/2) \} \\
  \sigma_1
  &= \inf \{ j \ge \rho_1 : 
    S(j) \not\in \bB(2R') \} \\
  \rho_{i+1}
  &= \inf \{ j : \rho_i \le j \le \tau_{n'},\, 
    S(j) \in \bB(R'/2) \} \qquad i \ge 1 \\
  \sigma_{i+1}
  &= \inf \{ j \ge \rho_{i+1} : 
    S(j) \not\in \bB(2R') \}. }
Let $B_i := \{ \sigma_i < \xi_A \}$, and
let $\caF_i$ and $\caG_i$, respectively, denote the 
$\sigma$-algebras generated by events up to time $\rho_i$
and $\sigma_i$, respectively. Also let
$\caG_0$ be the $\sigma$-algebra generated by events 
up to time $\tau_2$. Since $A$ contains a 
path from $0$ to $\partial \bB(R')$ we have
$\Pr [ B_i | \caF_i ] \le c < 1$, $i \ge 1$. By considering 
the walk between $\sigma_{i-1}$ and $\rho_i$, we also have
\eqnst
{ \Pr [ \rho_i = \infty | \caG_{i-1} ] 
  \le \frac{C}{\log(n'/R')}
  \le \frac{C'}{\log((\sqrt{\lambda}R')^{-1})}. }
Write 
\eqnsplst
{ F_i 
  &= \cap_{\ell=1}^{i-1} \{ \rho_\ell < \infty,\, B_\ell \} 
     \cap \{ \rho_i < \infty \} \\
  G_i
  &= \cap_{\ell=1}^i \{ \rho_\ell < \infty,\, B_\ell \}. } 
Hence we deduce
\eqnspl{e:lastbit}
{ \Pr [ S[\tau_2,\tau_{n'}] \cap A = \es ]
  &\le \Pr \left[ \cup_{i=1}^\infty \left[ \left(
      \cap_{\ell=1}^{i-1} \{ \rho_\ell < \infty,\, B_\ell \} 
      \right) 
      \cap \{ \rho_i = \infty \} \right] \right] \\
  &\le \sum_{i=1}^\infty \left( \prod_{\ell=1}^{i-1} 
      \Pr [ \rho_\ell < \infty | G_{\ell-1} ]
      \Pr [ B_\ell | F_\ell] \right)
      \Pr [ \rho_i = \infty | G_{i-1} ] \\
  &\le \sum_{i=1}^\infty c^{i-1} \frac{C}{\log(n'/R')} \\
  &\le \frac{C'}{\log((\sqrt{\lambda}R')^{-1})}. }
Using the Strong Markov property, we can combine the
bounds \eqref{e:inwards}, \eqref{e:outwards} and 
\eqref{e:lastbit}, and together with 
\eqref{e:separated} we deduce \eqref{e:2Dreturn}.

Finally, letting $\lambda \to 0$ we obtain 
the second statement of the Proposition, 
as the bounds are uniform in $\lambda$.
\end{proof}

We are ready to prove the analogue of 
Lemma \ref{lem:LEnotreturn} in the $d = 2$ case.

\begin{proposition}
\label{prop:2Dagree}
Assume $d = 2$. Let $R$, $n$ and $\lambda$ 
satisfy the relations as in 
Proposition \ref{prop:convLERW}. Let 
$m$, $R'$ and $\lambda$ satisfy the relations 
as in Proposition \ref{prop:2Dnotreturn}.
There exists a coupling between
$\LE(S[0,T])$ and $\hat{S}[0,\infty)$, 
such that if $R = 4 (R')^2$, then 
\eqnspl{e:2Dagree}
{ &\Pr \left[ \parbox{6.5cm}{$\LE(S[0,T]) \cap \bB(m) =
      \hat{S}^\lambda[0,R] \cap \bB(m)$; \\  
      $\hat{S}[0,\infty) \cap \bB(m) =
      \hat{S}[0,R] \cap \bB(m)$; \\
      $\hat{S}^\lambda[0,R] = \hat{S}[0,R]$} \right] \\
  &\qquad \geq 1 - O( (m/R') \log(R'/m) )
      - O( (\log R) R^{2-\beta} ) \\
  &\qquad\quad  - O( \lambda n^2 R^{1+\beta} / \log [ \lambda^{-1} ] )
      - O \left( \frac{R^2}{n} \log \frac{n}{R} \right). }
\end{proposition}

\begin{proof}
A self-avoiding walk of length $R = 4 (R')^2$ 
necessarily visits $\partial \bB(R')$. 
Consider the events in \eqref{e:2Dnotreturn1} and
\eqref{e:2Dnotreturn2}. On these events,
$\LE(S[0,T]) \cap \bB(m) = \hat{S}^\lambda[0,R] \cap \bB(m)$,
and $\hat{S}[0,\infty) \cap \bB(m) = \hat{S}[0,R] \cap \bB(m)$.
Due to Proposition \ref{prop:convLERW}, there
exists a coupling between 
$\hat{S}[0,R]$ and $\hat{S}^\lambda[0,R]$ such that
the two are identical with probability 
at least 
\eqnst
{ 1 - O( (\log R) R^{2-\beta} ) 
  - O( \lambda n^2 R^{1+\beta} / \log [ \lambda^{-1} ] )
  - O \left( \frac{R^2}{n} \log \frac{n}{R} \right). }
Hence we obtain the lemma.
\end{proof}

From this point on, the proof of Theorem \ref{thm:main}
is fairly similar to the
$d = 3$ case. Using the notation of 
Section \ref{sec:strategy}, the path $\pi^{(0)}_0$
is distributed as $\hat{S}[0,\infty)$, while
the path $\pi^{(\gamma)}_0$ is distributed as
$\LE(S[0,T])$. Proposition \ref{prop:2Dagree}
gives a coupling of the two paths. 
Let $z_1, \dots, z_N$ be a list of all vertices in 
$\partial_\ext B(k)$, and let $z_{N+1}, \dots, z_{N+M}$
be a list of all vertices in $B(k)$. Let 
$S^i$ be independent random walks started at $z_i$,
with geometric killing times $T^i$, 
$i = 1, \dots, N+M$. We use Wilson's algorithm 
with the walks $S^i$ to construct both
$\pi^{(\gamma)}_{z_i}$ and $\pi^{(0)}_{z_i}$.
Write $\xi^i_A$, $\tau^i_N$, etc. for hitting
and exit times associated with $S^i$.

Let $E_0$ denote the event on the left hand side
of \eqref{e:2Dagree}. Assume this event occurs, and
define
$E_i = \{ \xi^i_{\pi^{(0)}_0} < \tau^i_m, T^i > \tau^i_m \}$,
$i = 1, \dots, N$.
On the event 
$E_0 \cap (\cap_{i=1}^N E_i)$, the conditions (i)--(iii)
of Section \ref{sec:strategy} hold for the paths
$\pi^{(0)}_0, \pi^{(0)}_{z_1}, \dots, \pi^{(0)}_{z_N}$ and
$\pi^{(\gamma)}_0, \pi^{(\gamma)}_{z_1}, \dots, \pi^{(\gamma)}_{z_N}$.
The Beurling estimate and an argument similar to 
\eqref{e:smallT} gives 
\eqn{e:Eibound}
{ \Pr [E_i^c \,|\, E_0 ] 
  \le C (k/m)^{1/2} + C \lambda (\log m) m^2. }
Note that the union of the paths
$\pi^{(0)}_0 \cup (\cup_{i=1}^N \pi^{(0)}_{z_i})$ separate $B(k)$ from
$\partial \bB(m)$. Hence the walks started at
$z_{N+1}, \dots, z_{N+M}$ necessarily hit the
earlier paths before exiting $B(k)$. Let
$F_i = \{ T^i > \tau^i_m \}$, $i = N+1, \dots, N+M$.
On the event $E_0 \cap \left( \cap_{i=1}^N E_i \right) \cap
\left( \cap_{i=N+1}^{N+M} F_i \right)$, the required
coupling of paths is successful. Note that we have
\eqn{e:Fibound}
{ \Pr [ F_i^c \,|\, E_0 \cap (\cap_{i=1}^N E_i) ] 
  \le C \lambda (\log k) k^2. }
Combining \eqref{e:Eibound} and \eqref{e:Fibound}
we get the following bound for the right hand side 
of \eqref{e:mainthm}:
\eqnsplst
{ &\Pr [ E_0^c ] + \sum_{i=1}^N \Pr [ E_i^c \,|\, E_0 ]
      + \sum_{i=N+1}^{N+M} 
      \Pr [ F_i^c \,|\, E_0 \cap (\cap_{i=1}^N E_i) ] \\
  &\qquad \le C (m/R') \log (R'/m) + C (\log R) R^{2-\beta} 
      + C \lambda n^2 R^{1+\beta} / \log [ \lambda^{-1} ] \\
  &\qquad\quad + C (R^2/n) \log (n/R)
      + C (k^{3/2}/m^{1/2}) + C k \lambda (\log m) m^2
      + C (\log k) k^4 \lambda. } 
Note that since $k < R$ and $m < n$, the term 
$k \lambda (\log m) m^2$ is of smaller order than
the third term. Likewise, since $\beta > 2$, 
the term $(\log k) k^4 \lambda$
is of smaller order than the third term. Omitting
these terms we have the upper bound:
\eqnspl{e:prebound}
{  &C (m/R') \log (R'/m) + C (\log R) R^{2-\beta} 
      + C \lambda n^2 R^{1+\beta} / \log [ \lambda^{-1} ] \\
   &\qquad + C (R^2/n) \log (n/R)
      + C (k^{3/2} / m^{1/2}). } 

We now choose the parameters. Setting 
$m^{-1/2} = m (R')^{-1}$ will make 
the first term of the same order as the last term 
(up to a logarithm). Hence we will choose 
$R' = m^{3/2}$, and hence $R = C m^3$.
We also set $R^2 n^{-1} = m^{-1/2}$, which makes
the fourth term the same order as the last term
(up to a logarithm). Therefore we take
$n = m^{13/2}$. We set
$R^{2-\beta} = R^2 n^{-1}$, which makes the second 
term the same order as the fourth. Hence we
choose $\beta$ determined by the relation: 
$R^\beta = n$. Finally, we set
$\lambda n^2 R^{1+\beta} = \lambda n^3 R = m^{-1/2}$,
which makes the third term the same order as the last
one. This yields: $\lambda = m^{-23}$. 
Hence the optimal choice of $m$ in terms of $\lambda$
is $m = \lambda^{-1/23}$ ($\rho'' = 1/23$). 
This determines the other parameters as: 
$n = \lambda^{-13/46}$ ($\rho = 13/46$), 
$R = 4 \lambda^{-3/23}$ ($\rho' = 3/23$), 
$\beta = 13/2$, and
$R' = \lambda^{-3/46}$. We need 
$2k < m = \lambda^{-1/23}$, and hence
$\lambda \le \lambda_0 := (2k)^{-23}$.
With these choices the required relations 
between the parameters are satisfied: 
$\rho'' < \rho'/2$, $\rho' < \rho < 1/2$ and
\eqnst
{ \beta > 2; \qquad
  1 - 2 \rho > (1+\beta) \rho'; \qquad
  2 \rho' < \rho; \qquad
  (\rho'/2)4 < 1. }
Hence Propositions \ref{prop:convLERW}, 
\ref{prop:2Dnotreturn} and \ref{prop:2Dagree}
apply, and the bound in \eqref{e:prebound}
reduces to $C k^{3/2} \lambda^{1/46-o(1)}$.

The better upper bound is obtained by
setting each term equal to $k^{3/2} m^{-1/2}$.
This yields: $R' = k^{-3/2} m^{3/2}$, 
$R = 4 k^{-3} m^3$, $n = m^{13/2} k^{-15/2}$,
$R^\beta = n$ and $\lambda = k^{27} m^{-23}$. 
Hence the optimal choice of the parameters 
in terms of $\lambda$ and $k$ is: 
$m = \lambda^{-1/23} k^{27/23}$,
$R' = \lambda^{-3/46} k^6$,
$R = 4 \lambda^{-3/23} k^{12}$,
$n = \lambda^{-13/46} k^{3/23}$. 
The restrictions on the parameters are satisfied 
as follows: $2k < m$, $16m < R'$, $R < n < \lambda^{-1/2}$
are automatic for $\lambda \le \lambda_1$, with 
$\lambda_1$ independent of $k$.
The condition $\beta > 2$ can be satisfied if
\eqnst
{ \frac{\frac{13}{46} \log(1/\lambda) 
    + \frac{3}{23} \log k}{\log 4 
    + \frac{3}{23} \log(1/\lambda)
    + 12 \log k}
  > 2, }
which holds if $\lambda < 16^{-46} k^{-46(24-\frac{3}{23})}$.
So we can take $C_0 > 46 (24 - \frac{3}{23})$ and
$c_0 = 16^{-46}$. The requirements $1-2\rho < (1+\beta) \rho'$,
$2 \rho' < \rho$ and $(\rho'/2)4 < 1$ 
are automatically satisfied by the choice of the
exponents. Hence for $\lambda \le \lambda_0 := c_0 k^{-C_0}$ 
we get the upper bound $C k^{3/2} \lambda^{1/46-o(1)} k^{-27/46}$.
This proves the theorem in the $d = 2$ case.
\qed

\medbreak

\noindent{\bf Acknowledgements.}
I thank Frank Redig and Ellen Saada for helpful comments.

\end{document}